\documentclass[11pt]{article}

\usepackage[margin=1in]{geometry}
\usepackage{upquote}
\usepackage{amsfonts,amsmath,amssymb,amsthm}
\usepackage{longtable}
\usepackage{xcolor}
\usepackage{booktabs}
\usepackage{multirow}
\usepackage{hyperref}

\usepackage{graphicx}
\graphicspath{ {images/}}
\usepackage{epstopdf}
\usepackage{caption}
\usepackage{subcaption}
\newtheorem{theorem}{Theorem}[section]
\newtheorem{lemma}{Lemma}[section]
\newtheorem{problem}{Problem}
\theoremstyle{definition}

\numberwithin{equation}{section}

\title{\textbf{Robust numerical schemes for singularly perturbed delay parabolic convection diffusion problems with degenerate coefficient }}
\author{Pratima Rai\thanks{ Department of Mathematics, University of Delhi, Delhi-110007, India.
		({\tt prai@maths.du.ac.in}).} ~\thanks{Corresponding author.}
	,~	Swati Yadav\thanks{Department of Mathematics,
		University of Delhi, Delhi-110007, India.
		({\tt swatiyadav3317@gmail.com}).}}
\date{\today}

\begin{document}
	
	\maketitle
	
	\begin{abstract}
		This article studies a dirichlet boundary value problem for singularly perturbed time delay convection diffusion equation with degenerate coefficient. A priori explicit bounds are established on the solution and its derivatives. For asymptotic analysis of the spatial derivatives the solution is decomposed into regular and singular parts. To approximate the solution a numerical method is considered which consists of  backward Euler scheme for time discretization on uniform mesh and a  combination of midpoint upwind and central difference scheme for the spatial discretization on modified Shishkin mesh.   Stability analysis is carried out,  numerical results are presented and comparison is done with upwind scheme on uniform mesh as well as  upwind scheme on Shishkin mesh to  demonstrate the effectiveness of the proposed method.
		The convergence obtained in practical satisfies the theoretical predictions.

		\textbf{MSC classification 2010:} 65N12, 65N30, 65N06, 65N15.
	\end{abstract}
	
	\section{Introduction}
	Singularly perturbed parabolic delay differential equations (SPPDDEs) plays a crucial role in mathematical modeling of various real life phenomena which takes into consideration the past history of the system along with its present state. The delay or lag represent incubation period, gestation time, transport delays etc. The solution and dynamics of singularly perturbed delay partial differential equations are completely different from those of the partial differential equations without time delay.
	
	A characteristic example of SPPDDEs is the following equation arising in numerical control modeling a furnace
	used to process metal sheets~\cite{ansari}
	\begin{align}
	(\partial_t u-\varepsilon \partial_x^2u)(x,t)=
	v(g(u(x, t-\tau)))\partial _xu(x,t)+c[f(u(x,t-\tau))-u(x,t)],\quad (x,t)\in D.
	\end{align}
	Here, $u$ is the temperature distribution of the metal sheet, which is moving with velocity $v$ and heated by a source given by function $f$; both $v$ and $f$ are adapted dynamically by a controlling device monitoring the current temperature distribution. Since speed of the controller is finite it induces a fixed delay of length $\tau$. 
	
	Last one decade has witnessed a growing interest in the numerical study of singularly perturbed  parabolic delay differential equations (SPPDDEs). However, uniformly convergent
	numerical methods  are not much developed for SPPDDEs. Numerical study of SPPDDEs for the class of reaction diffusion equations was initiated by Ansari et. al. \cite{ansari}.  The authors used classical finite difference scheme on piecewise uniform Shishkin mesh.  Gowrisankar and Natesan ~\cite{gowrishankar2} used layer adapted meshes obtained via equidistributing a monitor function for the numerical solution of singularly perturbed parabolic delay differential reaction diffusion problems.  Sunil and Mukesh~\cite{sunil1} constructed a hybrid scheme consistng of HODIE type on generalized Shishkin mesh in spatial direction and implicit Euler scheme on uniform mesh in time direction for the numerical approximation of  singularly perturbed parabolic delay differential reaction diffusion problems. Joginder et al.~\cite{Joginder} designed and analyzed a domain decomposition method for the numerical solution of SPPDDEs. 
	
	For work on convection diffusion problem for SPPDDEs one can refer to \cite{das1,das2,gowrishankar1, Kaushik1, Kaushik2, Salama1}.  Aditya and Manju~\cite{Kaushik1} analyzed the weighted difference approximations on piecewise uniform mesh for singularly perturbed delay differential convection diffusion problems and established that the proposed scheme is $L_2^h$ stable.  Gowrisankar and Natesan~\cite{gowrishankar1} used layer adapted meshes based on equi-distribution of a monitor function for the  numerical solution of SPPDDEs of convection diffusion type. Abhishek and Natesan~\cite{das1} proposed a hybrid scheme on Shishkin mesh for the numerical solution of convection diffusion problem for SPPDDEs which is almost second order accurate in space and first order in time direction. In \cite{das2} the authors applied Richardson extrapolation on simple upwind scheme to obtain almost second order of convergence in space direction and second order of convergence in time direction for SPPDDEs of convection diffusion type. The authors~\cite{Salama1}  derived a higher order uniformly convergent method which is second order accurate in time and fourth order accurate in space for the numerical solution of singularly perturbed parabolic delay convection diffusion problems.    
	
	To the best of our knowledge, all the literature on the numerical solution of SPPDDEs of convection type  is restricted to the case when the convection coefficient has same sign throughout the domain. Hence, a very first attempt has been made here to construct a parameter uniform numerical scheme for such a class of problem with degenerating convection coefficient.    
	
	We consider the following problem on a rectangular domain:
	\begin{align}
	\label{1.1}
	L_{\varepsilon}u(x,t)= \left( \varepsilon \frac{\partial^2{u}}{\partial{x^2}}+a \frac{\partial{u}}{\partial{x}}-b\frac{\partial{u}}{\partial{t}}-c u\right) (x,t)= e(x,t) u(x, t- \tau)+f(x,t),
	\end{align}
	where $0<\varepsilon \leq 1$, $\tau > 0$, $(x,t)\in Q = \Omega \times (0,T] = (0,1) \times (0,T]$, $\overline{Q}=[0,1] \times [0,T]$, $T$ is some finite time such that $T=k \tau$ for some integer $k>1$, $ \Gamma =  \Gamma_{b}\cup \Gamma_{l}\cup \Gamma_{r}$, with the interval and boundary conditions given by
	\begin{align}
	\label{1.2}
	u(x,t) &= s(x,t)\hspace{0.3cm} on \hspace{0.3cm} \Gamma_{b} = \lbrace (x,t) : 0\leq x \leq 1, ~ -\tau \leq t \leq 0 \rbrace,\nonumber \\
	u(0,t) &= q_{0}(t)\hspace{0.3cm} on \hspace{0.3cm} \Gamma_{l} = \lbrace (0,t) : 0\leq t \leq  T\rbrace,\nonumber \\
	u(1,t) &= q_{1}(t)\hspace{0.3cm} on \hspace{0.3cm} \Gamma_{r} = \lbrace (1,t) : 0\leq t\leq T\rbrace.
	\end{align}
	The coefficients $a(x,t)$, $ b(x,t)$, $ d(x,t) $ and $ f(x,t) $ are sufficiently smooth functions such that 
	\begin{align}\label{1.3}
	a(x,t) &= a_{0}(x,t)x^p ,\hspace{0.3cm} p\geq 1, \hspace{0.3cm} \forall~ (x,t)\in \overline{Q},\nonumber \\
	a_0(x,t) &\geq \alpha > 0, \hspace{0.3cm} \forall~ (x,t) \in \overline{Q},\nonumber \\
	b(x,t) &\geq \beta > 0 , \hspace{0.3cm} \forall~ (x,t) \in \overline{Q},\nonumber \\
	c(x,t) & \geq \gamma > 0 ,\hspace{0.3cm}\forall~ (x,t) \in \overline{Q}, \nonumber \\
	e(x,t) & \geq 0 ,\hspace{0.3cm}\forall~ (x,t) \in \overline{Q}.
	\end{align}
	Problem (\ref{1.1})-(\ref{1.3}) covers the multiple turning points for $p>1$. The solution of the problem $(\ref{1.1})$-$(\ref{1.3})$ exhibit a parabolic boundary layer of width $O(\sqrt{\varepsilon})$ in the neighbourhood of the left boundary $\Gamma_l$ as all the characteristic curves of the reduced problem are parallel to the boundary $\Gamma_{r}$. \newline
	The existence of the unique solution of the Problem (\ref{1.1})-(\ref{1.3}) is guaranteed under the assumption that the problem data is H$\ddot{o}$lder continuous, sufficiently smooth and satisfy appropriate compatibility conditions at the corner points (0,0), (1,0), $(0,-\tau)$ and $(1,-\tau)$.\\
	In this article our focus is to develop a higher order robust numerical scheme for the solution of SPPDDEs with multiple degeneracy. The article is designed as follows. In Section 2, analytical aspects of the continuous problem are discussed and a priori estimates are established on the exact solution and its derivatives. In an attempt to design a higher order scheme, in Section 3, the considered problem is discretized by the hybrid scheme on a piecewise-uniform modified Shishkin mesh in the space direction and the implicit Euler method on uniform mesh in the time direction. Stability and error analysis have been carried out for the proposed scheme to establish $\varepsilon$-uniform convergence of $O(N^{-2}L^{2} + \Delta t)$. In Section 4, we combine the hybrid scheme with the Richardson extrapolation to increase the order of convergence from $O(N^{-2}L^{2}+ \Delta t)$ to $O(N^{-2}L^{2}+ (\Delta t)^2)$. Numerical experiments are conducted in Section 5 to verify the theoretical results and illustrate the efficiency of the proposed schemes as compared to upwind scheme on uniform mesh as well as upwind scheme on Shishkin mesh.\\
	\textbf{Notations:} Throughout this article, we use C as a generic  positive constant independent of $\varepsilon$ and the mesh parameters. All the functions defined on a domain Q are measured in supremum norm, denoted by
	\begin{align*}
	\| f \|_{\overline{Q}} = \sup \limits_{x \in \overline{Q}} |f(x)|.
	\end{align*}      

\section{A Priori Bounds}
In this section, a priori bounds for the solution $u(x,t)$ of the problem (\ref{1.1})-(\ref{1.3}) and its derivatives are estimated on the domain $\overline{Q}$. We derive some a priori bounds using the method of steps and the minimum principle for the opertaor $L_{\varepsilon}$. The delay term $u(x,t-\tau)$ is a known function $s(x,t-\tau)$ for $(x,t-\tau) \in [0,1] \times [0,\tau]$ and hence the RHS of (\ref{1.1}) becomes $e(x,t) s(x,t-\tau) + f(x,t)$. This gives us the solution $u(x,t)$ for $(x,t) \in [0,1] \times [0,\tau]$. Using this we can compute the solution $u(x,t)$ for $(x,t)\in [0,1] \times [\tau,2 \tau]$ and so on. Hence, using the method of steps the existence and uniqueness results can be established for all $(x,t) \in \overline{Q}$. The operator $L_{\varepsilon}$ satisfies the following minimum principle.
\begin{lemma}[Minimum Principle]
	\label{lm2.1}
	Let $w \in C^{2,1}(\overline{Q})$. If $w(x,t) \geq 0$, $\forall~ (x,t) \in \Gamma$ and $L_{\varepsilon}w(x,t) \leq 0$, $\forall ~ (x,t) \in Q$ then $w(x,t) \geq 0$, $\forall (x,t) \in \overline{Q}$.
\end{lemma} 
\begin{proof}
	The proof follows easily from \cite{MR0219861}.
\end{proof}

\begin{lemma}
	\label{lm2.2}
	Let $u(x,t)$ be the solution of the problem (\ref{1.1})-(\ref{1.3}) then for all $\varepsilon>0$ the following bound holds
	\begin{align*}
	\|u\|_{\overline{Q}} \leq \|u\|_{\Gamma} + \dfrac{T}{\beta} \|f\|_{\overline{Q}}.
	\end{align*}
\end{lemma}
\begin{proof}
	Using the barrier function 
	\begin{align*}
	\psi^{\pm}(x,t) = \|u\|_{\Gamma} + \frac{t}{\beta} \|f\|_{\overline{Q}} \pm u(x,t),
	\end{align*}
	the desired estimate can be obtained using the minimum principle.
\end{proof}

The problem data are assumed to be sufficiently smooth that guarantee the required smoothness of the solution on the set $\overline{Q}$.
We assume that the data of the problem (\ref{1.1})-(\ref{1.3}) satisfy the following conditions
\begin{align}
\label{r1}
& a,~b,~c,~f \in C_{\lambda}^{l,l/2}(\overline{Q}),~ s(x,t) \in C^{l+2,l/2+1}_{\lambda}(\Gamma_{b}), \nonumber \\
&q_{0}(t) \in C_{\lambda}^{l/2+1}(\overline{\Gamma}_0),~q_{1}(t) \in C_{\lambda}^{l/2+1}(\overline{\Gamma}_{1}),~ l\geq 0,~ \lambda \in (0,1).
\end{align}
Also, the data of the problem (\ref{1.1})-(\ref{1.3}) satisfy on $\Gamma^{c}=(\overline{\Gamma}_{0} \cup \overline{\Gamma}_1) \cap \Gamma_b$ ( i.e. the corner points $(0,0)$, $(1,0)$, $(0,-\tau)$ and $(1,-\tau)$ ) the compatibility conditions for the derivatives in $t$ upto order $K_0=[l/2]+1$. In the case when the initial function $s(x,t)$ together with its derivatives vanish on the set $\Gamma^c$, the following conditions,
\begin{align}
\label{r2}
&\dfrac{\partial^{k+k_0} s(x,t)}{\partial x^k \partial t^{k_0}}=0, ~ \dfrac{\partial ^{k_0} q_0(t)}{\partial t^{k_0}}=0,~
\dfrac{\partial ^{k_0} q_1(t)}{\partial t^{k_0}}=0, ~ 0 \leq k+2k_0 \leq l+2,\\
\text{and}~ &\dfrac{\partial^{k+k_0} f(x,t)}{\partial x^k \partial t^k_0}=0,~ 0 \leq k+2k_0 \leq l, ~ (x,t)\in S^c, \nonumber
\end{align}
guarantee the compatibility conditions for the derivatives in $t$ upto order $K_0=[l/2]+1$. These compatibility conditions ensure the existence of the unique solution $u(x,t) \in C^{K,K/2}_\lambda (\overline{Q})$, where $K=l+2$, for the problem (\ref{1.1})-(\ref{1.3}) \cite{MR0241822}.\newline

\begin{lemma}
	\label{lm2.3}
	Let the solution $u(x,t)$ of the problem (\ref{1.1})-(\ref{1.3}) satisfying the assumptions (\ref{r1})-(\ref{r2}) for $K=6$, then 
	\begin{align*}
	\left\| \dfrac{\partial ^{i+j} u}{ \partial x^i \partial t^j } \right\|_{\overline{Q}} \leq C \varepsilon^{-i/2}, \quad \forall~ 0 \leq i+2j \leq 6,
	\end{align*}
	where C is independent of $\varepsilon$.
\end{lemma}
\begin{proof}
	Using the method of steps we derive the bounds on the derivatives of the solution $u(x,t)$. First, we consider the case for $t \leq \tau$. Since, $u(x,t-\tau)$ is a known function in $[0,1] \times [0,\tau]$, the problem (\ref{1.1})-(\ref{1.3}) becomes
	\begin{align*}
	L_{\varepsilon}u(x,t)&= \left( \varepsilon \frac{\partial^2{u}}{\partial{x^2}}+a \frac{\partial{u}}{\partial{x}}-b\frac{\partial{u}}{\partial{t}}-c u\right) (x,t)= e(x,t) s(x, t- \tau)+f(x,t),~ \forall~ (x,t) \in Q_1=(0,1) \times (0,\tau],\\
	u(x,0) &= s(x,0)\hspace{0.3cm} on \hspace{0.3cm} \Gamma_{0} = \lbrace (x,t) : 0\leq x \leq 1, ~  t = 0 \rbrace,\nonumber \\
	u(0,t) &= q_{0}(t)\hspace{0.3cm} on \hspace{0.3cm} \Gamma_{l} = \lbrace (0,t) : 0\leq t \leq  \tau \rbrace,\nonumber \\
	u(1,t) &= q_{1}(t)\hspace{0.3cm} on \hspace{0.3cm} \Gamma_{r} = \lbrace (1,t) : 0\leq t\leq \tau \rbrace.
	\end{align*}
	The bounds in the interval $Q_1$ are obtained as follows. The variable $x$ is transformed to the stretched variable $\widetilde{x}= x/ \sqrt{\varepsilon}$, we write the problem (\ref{1.1})-(\ref{1.3}) as
	\begin{align}\label{2.1}
	\left(\widetilde{u}_{\widetilde{x}\widetilde{x}}+\widetilde{a} \varepsilon^{\frac{p-1}{2}} \widetilde{u}_{\widetilde{x}}-\widetilde{b}\widetilde{u}_{t}-\widetilde{c}\widetilde{u}\right)(\widetilde{x},t) &= \widetilde{e}(\widetilde{x},t) \widetilde{s}(\widetilde{x},t-\tau) + \widetilde{f}(\widetilde{x},t)\quad \mbox{in} \quad \widetilde{Q_1}, \\
	\widetilde{u}(\widetilde{x},t) &= \widetilde{s}(\widetilde{x},t)\hspace{0.3cm} on \hspace{0.3cm} \widetilde{\Gamma}_{0} = \lbrace (\widetilde{x},t) : 0\leq \widetilde{x} \leq 1/\sqrt{\varepsilon},~ t=0 \rbrace,\nonumber \\
	\widetilde{u}(0,t) &= \widetilde{q}_{0}(t)\hspace{0.3cm} on \hspace{0.3cm} \widetilde{\Gamma}_l = \lbrace (0,t) : 0\leq t \leq  \tau \rbrace,\nonumber \\
	\widetilde{u}(1/\sqrt{\varepsilon},t) &= \widetilde{q}_{1}(t)\hspace{0.3cm} on \hspace{0.3cm} \widetilde{\Gamma}_{r} = \lbrace (1/\sqrt{\varepsilon},t) : 0\leq t\leq \tau \rbrace, \nonumber
	\end{align}
	where $ \widetilde{Q}_1 = (0, 1/\sqrt{\varepsilon})\times (0,\tau] $ and $ \widetilde{\Gamma} $ is the boundary analogous to $ \Gamma $. Since, for $p>1$ the term $ \varepsilon^{\frac{p-1}{2}} $ is very small it can be neglected and for $ p=1 $ its value is one, the differential equation (\ref{2.1}) can be made independent of $ \varepsilon $. Using \cite[estimate (10.5)]{MR0241822} we have, for all non negative integers $i, j$ such that $0 \leq i + 2j \leq 4$, and all $\widetilde{N}_{\delta}$ in $\widetilde{D}_{\varepsilon}$,
	\begin{align}
	\label{2.2}
	\left\| \dfrac{\partial^{i+j} \widetilde{u}}{\partial x^i \partial t^j} \right\|_{\widetilde{N}_{\delta}} \leq C (1 + \|\widetilde{u}\|_{\widetilde{N}_{2\delta}}).
	\end{align}	
	Here, $C$ is a constant independent of ${\widetilde{N}_{\delta}}$ and for any $\delta>0$, ${\widetilde{N}_{\delta}}$ is a neighbourhood of diameter $\delta > 0$ in $\widetilde{Q}_1$. Transforming back to the original variable $x$, we get
	\begin{align*}
	\left\|\frac{\partial^{i+j}{u}}{\partial{x^i}\partial{t^j}} \right\|_{\overline{Q}_1} \leq C \varepsilon^{-i/2}( 1 + \|u\|_{\overline{Q}_1}).
	\end{align*}
	Using the bounds of $u(x,t)$ given in Lemma \ref{lm2.2} we get the desired estimates.\\
	Next, we consider the case for $t \in [\tau, 2 \tau]$. In this case $u(x,t)$ is the solution of the following initial boundary value problem (IBVP):
	\begin{align*}
	L_{\varepsilon}u(x,t)&= \left( \varepsilon \frac{\partial^2{u}}{\partial{x^2}}+a \frac{\partial{u}}{\partial{x}}-b\frac{\partial{u}}{\partial{t}}-c u\right) (x,t)= e(x,t) u(x, t- \tau)+f(x,t),~ \forall~ (x,t) \in Q_2=(0,1) \times (\tau,2\tau],\\
	u(x,\tau) &= s(x,\tau)\hspace{0.3cm} on \hspace{0.3cm} \Gamma_{0} = \lbrace (x,t) : 0\leq x \leq 1, ~  t = \tau \rbrace,\nonumber \\
	u(0,t) &= q_{0}(t)\hspace{0.3cm} on \hspace{0.3cm} \Gamma_{l} = \lbrace (0,t) : \tau \leq t \leq  2\tau \rbrace,\nonumber \\
	u(1,t) &= q_{1}(t)\hspace{0.3cm} on \hspace{0.3cm} \Gamma_{r} = \lbrace (1,t) : \tau \leq t\leq 2\tau \rbrace.
	\end{align*}
	Again the RHS is a known function so the proof follows on the similar lines as discussed in the case for $[0,\tau]$. We proceed similarly to prove the result for $t \in [0,T]$.
\end{proof}
The bounds obtained in the Lemma \ref{2.3} are not sufficient for proving $\varepsilon$-uniform error estimates. Therefore, stronger bounds on these derivatives are obtained by decomposing the solution $u(x,t)$ into the regular part $y(x,t)$ and the singular part $z(x,t)$. We define
\begin{align*}
u(x,t)= y(x,t) + z(x,t).
\end{align*}
The regular component $y(x,t)$ is further decomposed into the sum
\begin{align*}
y=(y_0 + \sqrt{\varepsilon} y_1 + \varepsilon y_2 + \varepsilon^{3/2} y_3)(x,t),
\end{align*}
where $y_0$, $y_1$, $y_2$ and $y_3$ are defined as 
\begin{align}
\label{2.3}
a(x,t) \dfrac{\partial y_0}{\partial x} - b(x,t) \dfrac{\partial y_0}{\partial t} - c(x,t) y_0 &= e(x,t) y_0(x,t-\tau) + f(x,t), \quad \forall ~ (x,t) \in Q,\nonumber \\ 
y_0(x,t)&=u(x,t), \quad \forall ~ (x,t) \in \Gamma_b \cup \Gamma_r,
\end{align}
\begin{align}
\label{2.4}
a(x,t) \dfrac{\partial y_1}{\partial x} - b(x,t) \dfrac{\partial y_1}{\partial t} - c(x,t) y_1 &=e(x,t) y_1(x,t-\tau) -\sqrt{\varepsilon} \dfrac{\partial^2 y_0}{\partial x^2} , \quad \forall ~ (x,t) \in Q,\nonumber \\ 
y_1(x,t)&=0, \quad \forall ~ (x,t) \in \Gamma_b \cup \Gamma_r,
\end{align}
\begin{align}
\label{2.5}
a(x,t) \dfrac{\partial y_2}{\partial x} - b(x,t) \dfrac{\partial y_2}{\partial t} - c(x,t) y_2 &=e(x,t) y_2(x,t-\tau) -\sqrt{\varepsilon} \dfrac{\partial^2 y_1}{\partial x^2} , \quad \forall ~ (x,t) \in Q,\nonumber \\ 
y_2(x,t)&=0, \quad \forall ~ (x,t) \in \Gamma_b \cup \Gamma_r
\end{align}
and
\begin{align}
\label{2.6}
L_{\varepsilon} y_3 &=e(x,t) y_3(x,t-\tau) -\sqrt{\varepsilon} \dfrac{\partial^2 y_2}{\partial x^2} , \quad \forall ~ (x,t) \in Q,\nonumber \\
y_3(x,t)&=0, \quad \forall ~ (x,t) \in \Gamma.
\end{align}
Therefore, the regular component $y(x,t)$ satisfies 
\begin{align}
\label{2.7}
L_{\varepsilon}y(x,t) &= f(x,t) + e(x,t) y(x,t-\tau), \quad \forall ~ (x,t) \in Q,\nonumber \\
y(x,t)&=u(x,t), \quad \forall ~ (x,t) \in \Gamma_b \cup \Gamma_r, \nonumber \\
y(x,t)&=y_0(x,t) + \sqrt{\varepsilon}y_1(x,t) + \varepsilon y_2(x,t) + \varepsilon^{3/2} y_3(x,t), \quad \forall ~ (x,t) \in \Gamma_l.
\end{align}
The singular component $z(x,t)$ satisfies the following IBVP: 
\begin{align}
\label{2.8}
L_{\varepsilon}z(x,t) &=  e(x,t) z(x,t-\tau), \quad \forall ~ (x,t) \in Q,\nonumber \\
z(x,t)&=0, \quad \forall ~ (x,t) \in \Gamma_b \cup \Gamma_r,\nonumber \\
z(x,t)&= u(x,t) - y(x,t), \quad \forall ~ (x,t) \in \Gamma_l.
\end{align}

\begin{theorem}\label{thm1}
	For all non-negative integers $ i, j $ such that $ 0\leq i+2j\leq 6, $ the regular component $ y(x,t) $ satisfies
	\begin{align*}
	\left\| \frac{\partial^{i+j}{y}}{\partial{x^i}\partial{t^j}}\right\|_{\overline{Q}}\leq C\left(1+\varepsilon^{\frac{3-i}{2}}\right)
	\end{align*} 
	and the singular component $ z(x,t) $ satisfies
	\begin{align*}
	\left\| \frac{\partial^{i+j}{z}}{\partial{x^i}\partial{t^j}}\right\|_{\overline{Q}}\leq C\left(\varepsilon^{-i/2}\exp\left(\frac{-m x}{\sqrt{\varepsilon}}\right)\right),
	\end{align*} 
	where $m=\sqrt{\gamma}$.
\end{theorem}
\begin{proof}
	We first consider the interval $[0,1] \times [0,\tau]$. The data of the problems (\ref{2.3})-(\ref{2.7}) are assumed to be sufficiently smooth and satisfy the appropriate compatibility conditions to ensure the existence of the unique solution $y_0$, $y_1$, $y_2$, $y$ $\in C_{\lambda}^{K,K/2}$, for $K=6$. Since $ y_0$, $y_1$ and $y_2 $ are solutions of first order hyperbolic equations (\ref{2.3}), (\ref{2.4}) and (\ref{2.5}) over $Q_1$ as well we have the following estimates
	\begin{align}
	\label{2.9}
	\left\|\frac{\partial^{i+j}{y_0}}{\partial{x^i}\partial{t^j}}\right\|_{\overline{Q}_1} \leq C,
	\end{align}
	\begin{align}
	\label{2.10}
	\left\|\frac{\partial^{i+j}{y_1}}{\partial{x^i}\partial{t^j}}\right\|_{\overline{Q}_1} \leq C
	\end{align}
	and
	\begin{align}
	\label{2.11}
	\left\|\frac{\partial^{i+j}{y_2}}{\partial{x^i}\partial{t^j}}\right\|_{\overline{Q}_1} \leq C.
	\end{align}
	As $y_3$ is the solution of a problem similar to the initial boundary value problem (\ref{1.1})-(\ref{1.3}) therefore, for all non-negative integer $ i, j $ such that $ 0\leq i+2j \leq 4 $, we have 
	\begin{align}
	\label{2.12}
	\left\|\frac{\partial^{i+j}{y_3}}{\partial{x^i}\partial{t^j}}\right\|_{\overline{Q}} \leq C\varepsilon^{-i/2}.
	\end{align}
	Using inequalities (\ref{2.9})-(\ref{2.12}) we obtain the required estimates for the regular component $y(x,t)$ for $(x,t) \in \overline{Q}_1$.\newline
	To obtain the bound on the singular component, we define two barrier functions 
	\begin{align*}
	\psi^\pm(x,t) = C\exp\left(\frac{-m x}{\sqrt{\varepsilon}}\right)\exp(t)\pm z(x,t), \quad \forall \thinspace (x,t)\in \overline{Q}_1,
	\end{align*}
	where $ C $ is chosen sufficiently large such that we have 
	\begin{align*}
	\psi^\pm(x,t)\geq 0, \quad \forall \thinspace (x,t)\in \Gamma.
	\end{align*}
	Now,
	\begin{align*}
	L_\varepsilon \psi^\pm(x,t)&= C\exp\left(\frac{-m x}{\sqrt{\varepsilon}}\right)\exp(t)\left(m^2-\frac{a(x,t)m}{\sqrt{\varepsilon}}-b(x,t)-c(x,t)\right) \leq 0, \quad \forall ~(x,t)\in Q_1. \nonumber \\ 
	\end{align*}
	By Minimum principle, we have 
	\begin{align*}
	\left|z(x,t)\right| &\leq C \exp\left(\frac{-m x}{\sqrt{\varepsilon}}\right)\exp(t) \leq C \exp\left(\frac{-m x}{\sqrt{\varepsilon}}\right), \quad \forall \thinspace (x,t)\in \overline{Q}_1.
	\end{align*}
	To obtain the bound on the derivatives of $z$ we transform the variable $x$ to the stretched variable $ \widetilde{x} = \dfrac{x}{\sqrt{\varepsilon}}$. The transformed differential equation becomes independent of $\varepsilon$. For each neighbourhood $\widetilde{N}_{\delta}$ in $(2,1/\sqrt{\varepsilon}) \times (0,\tau)$ using \cite[\S 4.10]{MR0241822}, we have
	\begin{align}
	\label{2.13}
	\left\|\frac{\partial^{i+j}{\widetilde{z}}}{\partial{\widetilde{x}^i}\partial{t^j}} \right\|_{\widetilde{N}_\delta} \leq C \|\widetilde{z}\|_{\widetilde{N}_{2\delta}}.
	\end{align}
	The required bounds can be obtained by transforming the inequality (\ref{2.13}) in terms of the original variable $x$ and using the bound just obtained on $z(x,t)$. Similarly, for each neighbourhood $\widetilde{N}_{\delta}$ in $(0,2] \times (0,\tau)$ using \cite[\S 4.10]{MR0241822}, we have 
	\begin{align}
	\label{2.13a}
	\left\|\frac{\partial^{i+j}{\widetilde{z}}}{\partial{\widetilde{x}^i}\partial{t^j}} \right\|_{\widetilde{N}_\delta} \leq C ( 1 + \|\widetilde{z}\|_{\widetilde{N}_{2\delta}}).
	\end{align}
	Again transforming the inequality (\ref{2.13a}) in terms of the original variable $x$, using the bound on $z(x,t)$ and noting that $e^{-x/\sqrt{\varepsilon}}>C$ for $\widetilde{x}>2$, we have the required bounds. Next, consider the second interval $[\tau,2\tau]$. In this case we have $y(x,t)$ satisfy the problem (\ref{2.7}) for $(x,t) \in Q_2$. The argument for rest of the proof is same as in the first case. The proof for $t \geq 2\tau $ also follows on the same lines.  
\end{proof}
\section{Discrete Problem} 
In this section, we discretize  the problem (\ref{1.1})-(\ref{1.3}) in both space and time direction. Firstly, a modified Shishkin mesh $S(L)$ is constructed to discretize the spatial domain.\\
Let $\overline{\Omega}^N := \{x_i\}_{i=0}^N$ be the partition of the spatial domain $\overline{\Omega}$. We define the transition parameter $\sigma$ by
\begin{align*}
\sigma = \min \{1/2, \sigma_0 \sqrt{\varepsilon}L \},
\end{align*} 
where $L$ satisfies $\ln(\ln N) < L \leq \ln(N) $ and $e^{-L} \leq L/N$. The fitted piecewise uniform mesh $S(L)$ is constructed by dividing the domain $\overline{\Omega}$ into two subdomains $\overline{\Omega} = \overline{\Omega}_1 \cup \overline{\Omega}_2$, where ${\Omega}_1 = (0, \sigma]$ and $\Omega_2 = (\sigma,1)$. A piecewise uniform mesh $\Omega_{\sigma}^N$ on $\Omega$ with $N$ mesh points is obtained by placing a uniform mesh with $N/2$ mesh points in each subintervals.
The spatial step size $ h_i = x_i-x_{i-1}$, for $i= 1,2,\ldots,N $ is defined as
\begin{align*}
h_i = \begin{cases}
h = \dfrac{2\tau}{N}, \hspace{3cm} 1 \leq i \leq \dfrac{N}{2},\\\\
H = \dfrac{2(1-\tau)}{N}, \hspace{2cm}  \dfrac{N}{2}+1 \leq i \leq N,
\end{cases}
\end{align*}
where $ h $ and $ H $ are the spatial step size in $ [0,\tau] $ and $ (\tau,1], $ respectively.\newline
For temporal discretization a uniform mesh $\Omega^M$ and $\Omega^{m_\tau}$ with $M$ and $m_\tau$ mesh points is considered by placing a uniform mesh with $M$ and $m_\tau$ mesh points in $[0,T]$ and $[-\tau,0)$, respectively. The uniform step size $\Delta t$ in time direction satisfies $\tau=m_{\tau} \Delta t$, where $m_\tau$ is a positive integer, $t_n=n \Delta t$, $n \geq -m_\tau$. \\
Piecewise uniform tensor product meshes $Q_{\sigma}^{N,M}$ on $Q$ and $\Gamma_{b,\sigma}^{N,M}$ on $\Gamma_b$ are defined as
\begin{align*}
Q_{\sigma}^{N,M}= \Omega_{\sigma}^{N} \times \Omega^{M}, \quad \Gamma_{b,\sigma}^{N,M} = \Omega_{\sigma}^{N} \times \Omega^{m_{\tau}}
\end{align*}
and the boundary points $\Gamma_{\sigma}^{N,M}$ of $Q_{\sigma}^{N,M}$ are defined as $\Gamma_{\sigma}^{N,M}= \overline{Q}^{N,M} \cap \Gamma$. We put $\Gamma_{l,\sigma}^{N,M}= \overline{Q}^{N,M} \cap \Gamma_{l}$ and $\Gamma_{r,\sigma}^{N,M}= \overline{Q}^{N,M} \cap \Gamma_{r}$. For $\sigma= 1/2$, the mesh is uniform and for $\sigma= \sigma_0 \sqrt{\varepsilon}L$ the mesh points get condensed at the left side of the domain.

\subsection{The finite difference scheme}
For any mesh function $v_i^n= v(x_i,t_n)$, the forward, backward and central difference operators $D_x^{+}$, $D_x^{-}$, $D_x^{0}$ in space and $D_t^-$ in time are defined as
\begin{align*}
D_x^{+}v_{i}^n &= \dfrac{v_{i+1}^n-v_{i}^n}{h_{i+1}}, \quad D_x^{-}v_{i}^n= \dfrac{v_i^n-v_{i-1}^{n}}{h_i},\\
D_x^{0} v_i^n &= \dfrac{v_{i+1}^{n}-v_{i-1}^n}{\widehat{h}_i}, \quad D_{t}^{-}= \dfrac{v_i^n - v_i^{n-1}}{\Delta t},
\end{align*}
where $\widehat{h}_i=h_i +h _{i+1}$ for $i=1, \ldots, N-1$.
We also define the second-order finite difference operator $\delta_x^{2} v_i^n$ in space by
\begin{align*}
\delta_x^2 v_{i}^{n} = \dfrac{2(D_x^{+}v_i^n - D_x^{-}v_i^n)}{\widehat{h}_i}
\end{align*}
and $v_{i \pm 1/2}^{n} = \dfrac{v_{i \pm 1}^n + v_i^n}{2}$.
Applying the central difference scheme in the interval $I=\{i \in \{1,2,...,N-1\}, ~a_i^n h_i < 2 \varepsilon \}$ and the mid-point upwind scheme in the remaining region, we get the following discrete problem
\begin{align}
\label{3.1}
\begin{cases}
U_i^0=s(x_i,t_n), \quad for ~ (x_i,t_n) \in \Gamma_{b,\sigma}^{N,M},\\
L_{\varepsilon}^{N,M} \equiv
\begin{cases}
L_{\varepsilon,cen}^{N,M} U_i^n = f_i^n + e_i^n U(x_i,t_{n-m_{\tau}}), \quad for ~ i \in I,~ n \Delta t \leq T, \\
L_{\varepsilon,mu}^{N,M} U_i^n = f_{i+1/2}^n + e_{i+1/2}^n U(x_{i+1/2}, t_{n-m_{\tau}}), \quad for ~ i \notin I,~ n \Delta t \leq T, 
\end{cases}\\
U_{0}^{n}=q_0(t_n), \quad U_{N}^{n}=q_{1}(t_n), \quad for ~ n \geq 0,
\end{cases}
\end{align}
where
\begin{align*}
L_{\varepsilon,cen}^{N,M} U_i^n&= \varepsilon \delta^2 U_i^n + a_i^n D_x^{0} U_i^n - b_i^n D_t^{-}U_i^n - c_i^n U_i^n,\\
L_{\varepsilon,mu}^{N,M} U_i^n&= \varepsilon\delta_x^2 U_i^n + a_{i+1/2}^n D_x^{+} U_i^n - b_{i+1/2}^n D_t^{-} U_{i+1/2}^{n} - c_{i+1/2}^n U_{i+1/2}^n.
\end{align*}
On simplifying the terms in the system of Eqns. (\ref{3.1}), we obtain the following system of equations on the mesh $\overline{Q}_{\sigma}^{N,M}$ 
\begin{align}
\label{3.2}
\begin{cases}
U_i^0=s(x_i,t_n), ~ for ~ (x_i,t_n) \in \Gamma_{b,\sigma}^{N},\\
\begin{cases}
L_{\varepsilon}^{N,M} U_i^n &= \widetilde{f}_i^{n},~ for ~ 1 \leq i \leq N-1,~ 1\leq n \leq M\\
U_{0}^n&= q_0(t_n),~ U_{N}^{n}=q_1(t_n),~for ~ 1\leq n \leq M. 
\end{cases}
\end{cases}
\end{align}
where
\begin{align}
\label{3.3}
L_{\varepsilon}^{N,M} U_i^n := \begin{cases}
r_{cen,i}^{-} U_{i-1}^n + r_{cen,i}^{0} U_i^n + r_{cen,i}^{+}U_{i+1}^n, \quad for ~ i \in I,\\
r_{mu,i}^{-} U_{i-1}^n + r_{mu,i}^{0} U_{i}^{n} + r_{mu,i}^{+}U_{i+1}^{n}, \quad for ~i \notin I,
\end{cases}
\end{align}

\begin{align}
\widetilde{f}_i^n = \begin{cases}
\label{3.4}
[m_{cen,i}^- f_{i-1}^n + m_{cen,i}^0 f_i^n + m_{cen,i}^+ f_{i+1}^n]+[p_{cen,i}^{-}U_{i-1}^n + p_{cen,i}^0 U_{i}^n + p_{cen,i}^+ U_{i+1}^{n}] \\ + [q_{cen,i}^{-}U_{i-1}^{n-m_\tau} + q_{cen,i}^0 U_{i}^{n-m_\tau} + q_{cen,i}^+ U_{i+1}^{n-m_\tau}],\quad for ~ i \in I, \\\\
[m_{mu,i}^- f_{i-1}^n + m_{mu,i}^0 f_i^n + m_{mu,i}^+ f_{i+1}^n]+[p_{mu,i}^{-}U_{i-1}^n + p_{mu,i}^0 U_{i}^n + p_{mu,i}^+ U_{i+1}^{n}] \\ + [q_{mu,i}^{-}U_{i-1}^{n-m_\tau} + q_{mu,i}^0 U_{i}^{n-m_\tau} + q_{mu,i}^+ U_{i+1}^{n-m_\tau}],\quad for ~ i \notin I 
\end{cases}
\end{align} 
and various coefficients are given by
 \begin{align}
 \label{3.5}
\begin{cases}
r^{-}_{cen,i} &= \frac{2 \varepsilon \Delta t}{\widehat{h_i}h_{i}} - \frac{\Delta t a^{n}_{i}}{\widehat{h_i}},\\
r^{0}_{cen,i} &= \frac{-2 \varepsilon \Delta t}{\widehat{h_i}}\left( \frac{1}{h_{i}} + \frac{1}{h_{i+1}} \right) - b^{n}_{i} - \Delta t c^{n}_{i},\\
r^{+}_{cen,i} &= \frac{2 \varepsilon \Delta t}{\widehat{h_i}h_{i+1}} + \frac{\Delta t a^{n}_{i}}{\widehat{h_i}},\\
p_{cen,i}^{-}&=0, \quad p_{cen,i}^{0}=b_i^n, \quad p_{cen,i}^{+}=0,\\
m_{cen,i}^{-}&=0, \quad m_{cen,i}^{0}=\Delta t, \quad m_{cen,i}^{+}=0,\\
q_{cen,i}^{-}&=0, \quad q_{cen,i}^{0}= \Delta t e_i^n, \quad q_{cen,i}^{+}=0.
\end{cases}
 \end{align}
 and
 \begin{align}
 \label{3.6}
 \begin{cases}
r^{-}_{mu,i} &= \frac{2 \varepsilon \Delta t}{\widehat{h_i}h_{i}} ,\\
r^{0}_{mu,i} &= \frac{-2 \varepsilon \Delta t}{\widehat{h_i}}\left( \frac{1}{h_{i}} + \frac{1}{h_{i+1}} \right)  - \frac{ a^{n}_{i+1/2} \Delta t }{h_{i+1}} - \frac{b^{n}_{i+1/2}}{2 } - \frac{ c^{n}_{i+1/2} \Delta t}{2},\\
r^{+}_{mu,i} &= \frac{2\varepsilon \Delta t }{\widehat{h_i}h_{i+1}} + \frac{a^{n}_{i+1/2} \Delta t}{h_{i+1}} - \frac{b^{n}_{i+1/2}}{2 } - \frac{ c^{n}_{i+1/2} \Delta t}{2},\\
p_{mu,i}^{-}&=0, \quad p_{mu,i}^{0}=\dfrac{b_{i+1/2}^n}{2}, \quad p_{mu,i}^{+}=\dfrac{b_{i+1/2}^n}{2},\\
m_{mu,i}^{-}&=0, \quad m_{mu,i}^{0}=\Delta t/2, \quad m_{mu,i}^{+}=\Delta t/2,\\
q_{mu,i}^{-}&=0, \quad q_{mu,i}^{0}= \dfrac{\Delta t e_{i+1/2}^n}{2}, \quad q_{mu,i}^{+}=\dfrac{\Delta t e_{i+1/2}^n}{2}.
 \end{cases}
 \end{align}
 \textbf{Remark:} As $a(x_{N/2},t_n)>0$ and $a_0(x,t)\geq \alpha >0$ on $\overline{Q}$, we can conclude that there exists a constant $\kappa > 0$ such that $a(x_i,t_n) \geq \kappa >0$, for $N/2 \leq i \leq N$, $n \Delta t \leq T$.
 \begin{lemma}
 	\label{lm3.1}
 	Let $N_0$ be the smallest positive integer satisfying 
 	\begin{align}
 	\label{3.7}
 	N_0 \kappa \geq \frac{\Vert b \Vert_{\overline{Q}}}{\Delta t}+\Vert c \Vert_{\overline{Q}}, \quad 2 \tau_0 \Vert a_0 \Vert_{\overline{Q}} < \frac{N_0}{(\ln(N_0))^2} .
 	\end{align}
 	Then, for all $N \geq N_0$, we have
 	\begin{align*}
 	\begin{cases}
 	r_{cen,i}^{-} ,r_{cen,i}^{+}, r_{mu,i}^{-}, r_{mu,i}^{+} > 0, \quad 1 \leq i \leq N-1,\\
 	|r_{cen,i}^{-}|+|r_{cen,i}^{+}| < |r_{cen,i}^{0}|, \quad i \in I,\\
 	|r_{mu,i}^{-}|+|r_{mu,i}^{+}| < |r_{mu,i}^{0}|, \quad i \notin I,\\
 	| r_{cen,1}^{+} | < | r_{cen,1}^{0}| ,~  | r_{cen,N-1}^{-} | < | r_{cen,N-1}^{0}|, ~ \text{and} ~ | r_{mu,N-1}^{-} | < | r_{mu,N-1}^{0}|.
 	\end{cases}
 	\end{align*}
 \end{lemma}
\begin{proof}
	We first consider the case when $i \in I$. From (\ref{3.5}) we clearly have 
	\begin{align*}
	r^{-}_{cen,i} &= \frac{2 \varepsilon \Delta t}{\widehat{h_{i}}h_{i}} - \frac{\Delta t a_{i}^{n}}{\widehat{h_i}}
	=\frac{ \Delta t}{\widehat{h_i}}\left( \frac{2 \varepsilon}{h_i} - a_{i}^{n} \right) > 0
	\end{align*}
	and 
	\begin{align*}
	 r_{cen,i}^{+} = \dfrac{2 \varepsilon \Delta t}{\widehat{h_i}h_{i+1}} + \dfrac{\Delta t a_i^{n}}{\widehat{h_i}} > 0.
	\end{align*}
	Also, from (\ref{3.5}) we have 
	\begin{align*}
		\mid r_{cen,i}^{+} \mid + \mid r_{cen,i}^{-} \mid = \dfrac{2 \varepsilon \Delta t}{\widehat{h_i}} \left( \dfrac{1}{h_i} + \dfrac {1}{h_{i+1}} \right) < \mid r_{cen,i}^{0} \mid .
	\end{align*}
Next, we consider the case when $i \notin I$. For $N \geq N_0$, where $N_0$ satisfies
\begin{align*} 
2 \tau_0 \Vert a_0 \Vert_{\overline{Q}} < \frac{N_0}{(\ln(N_0))^2},
\end{align*}
we have $a_i^{n}h_i < 2 \varepsilon$, for $i=1,2,\ldots,N/2$. It can be clearly seen that $\{1,2,\ldots,N/2\} \subset I$, for $N \geq N_0$ and therefore, $L_{\varepsilon,mu}^{N}$ is applied for $i> N/2$ where $i \notin I$.\\
Clearly, $ r_{mu,i}^{-} = \dfrac{2 \varepsilon \Delta t}{\widehat{h_i} h_i} > 0 $ and
\begin{align*}
r_{mu,i}^{+}&= \frac{2 \varepsilon \Delta t}{\widehat{h_i}h_{i+1}} + \frac{a_{i+1/2}^{n} \Delta t}{h_{i+1}} - \frac{b_{i+1/2}^{n}}{2 } - \frac{c_{i+1/2}^{n} \Delta t}{2}\nonumber \\
&=\frac{2 \varepsilon \Delta t}{\widehat{h_i}h_{i+1}} + \Delta t \left( \frac{a_{i+1/2}^{n}}{h_{i+1}} - \frac{b_{i+1/2}^{n}}{2 \Delta t} - \frac{c_{i+1/2}^{n}}{2} \right)\\
&> \Delta t \left( \frac{a_{i+1/2}^{n}}{h_{i+1}} - \frac{b_{i+1/2}^{n}}{2 \Delta t} - \frac{c_{i+1/2}^{n}}{2} \right)= \Delta t \left( \frac{a_{i+1/2}^{n}}{H} - \frac{b_{i+1/2}^{n}}{2 \Delta t} - \frac{c_{i+1/2}^{n}}{2} \right).
\end{align*}
 Using $ \dfrac{1}{H} = \dfrac{N}{2(1-\tau)} \geq \dfrac{N_0}{2}$, we get
\begin{align*}
r_{mu,i}^{+} > \frac{\Delta t}{2} \left( N_0 \kappa - \frac{\|b\|_{\overline{Q}}}{\Delta t} - \| c \|_{\overline{Q}}   \right).
\end{align*}
Applying inequality (\ref{3.7}) and taking $N \geq N_0$ we have $r^{+}_{mu,i}>0$. Also, from (\ref{3.6}) we have 
\begin{align*}
|r_{mu,i}^{-}|+|r_{mu,i}^{+}|= \dfrac{2 \varepsilon \Delta t}{\widehat{h_i}h_{i}} +  \dfrac{2\varepsilon \Delta t }{\widehat{h_i}h_{i+1}} + \dfrac{a^{n}_{i+1/2} \Delta t}{h_{i+1}} - \dfrac{b^{n}_{i+1/2}}{2 } - \dfrac{ c^{n}_{i+1/2} \Delta t}{2} < |r_{mu,i}^{0}|.
\end{align*}
From (\ref{3.5}) and (\ref{3.6}) we can easily get  $\vert r_{cen,1}^{+} \vert < \vert r_{cen,1}^{0} \vert$, $ \vert r_{cen,N-1}^{-} \vert < \vert r_{cen,N-1}^{0} \vert $ and $ \vert r_{mu,N-1}^{-} \vert < \vert r_{mu,N-1}^{0} \vert $.
\end{proof}
Above Lemma establishes that the operator $L_{\varepsilon}^{N,M}$ satisfies the following discrete minimum principle
\begin{lemma}[Discrete Minimum Principle]
	\label{lm3.2}
	Let $W^N$ be any mesh function defined on $\overline{Q}_{\sigma}^{N,M}$. If $W^N(x_i,t_n) \geq 0$,~ $\forall~ (x_i,t_n) \in \Gamma_{\sigma}^{N,M}$ and $L_{\varepsilon}^{N,M}W^{N}(x_i,t_n) \leq 0$,~ $\forall~ (x_i,t_n) \in Q_{\sigma}^{N,M}$, then $W^{N}(x_i,t_n) \geq 0$,~ $\forall~ (x_i,t_n) \in \overline{Q}_{\sigma}^{N,M}$.	
\end{lemma}
\begin{lemma}
	\label{lm3.3}
	Let $ W^N $ be any mesh function defined on $ \overline{Q}_{\sigma}^{N,M} $.  If $ W^{N}(x_i,t_n)\geq 0, \thinspace \forall \thinspace (x_i,t_n)\in \Gamma_{\sigma}^{N,M} $ then
	\begin{align*}
	\left|W^{N}(x_i,t_n)\right|\leq \max\limits_{\Gamma_{\sigma}^{N,M}}|W^N|+\frac{T}{\beta}\max\limits_{Q_{\sigma}^{N,M}}|L_\varepsilon^{N,M} W^N| ,\quad \forall \thinspace (x_i,t_n)\in \overline{Q}_{\sigma}^{N,M}.
	\end{align*}
\end{lemma}
\begin{proof}
	Constructing the following barrier function
	\begin{align*}
	\psi^\pm(x_i,t_n) = \max\limits_{\Gamma^{N,M}}|W^N|+\frac{t_n}{\beta}\max\limits_{Q^{N,M}}|L_\varepsilon^{N,M} W^N|\pm W^{N}(x_i,t_n).
	\end{align*}
	and using the discrete minimum principle we get the desired estimate.
\end{proof}
\subsection{Error Analysis}
In this section, we provide error estimates for the regular and the singular component of the numerical solution separately. Finally, they are combined to provide parameter uniform error estimates for the proposed hybrid scheme. To prove $\varepsilon$-uniform convergence of the proposed scheme we consider the barrier function 
\begin{align}
\label{3.8}
\phi_i^n(\mu) = \begin{cases}
\prod\limits_{j=1}^i\left(1+\dfrac{\mu h_j}{\sqrt{\varepsilon}}\right)^{-1}, \quad i=1,2,\cdots,N,\\
1, \hspace{3cm} i=0,
\end{cases}
\end{align}
where $ \mu $ is a constant. Also,
\begin{align}
\label{3.9}
\begin{cases}
\phi_{i-1}^n(\mu) &=\left(1+\dfrac{\mu h_i}{\sqrt{\varepsilon}}\right) \phi_i^n(\mu), \quad i=1,2,\cdots,N,\\
\phi_{i+1}^n(\mu)&=\left(1+\dfrac{\mu h_{i+1}}{\sqrt{\varepsilon}}\right)^{-1} \phi_i^n(\mu), \quad i=1,2,\cdots,N-1.
\end{cases}
\end{align}

\begin{lemma}
	\label{lm3.4}
For each $ 0 \leq i \leq N$ and $ 0<\mu<\dfrac{m}{2} $, the barrier function $\phi_i^n(\mu)$ satisfies the following inequalities
	\begin{align}
	\label{3.10}
	L_\varepsilon^{N,M}\phi_i^n(\mu)\leq \begin{cases}
	\dfrac{-C}{\sqrt{\varepsilon}}\phi_i^n(\mu), \quad i\in I, \thinspace n\Delta t\leq T,\\\\
	\dfrac{-C}{\sqrt{\varepsilon}+\mu h_{i+1}}\phi_i^n(\mu), \quad i\not\in I, \thinspace n\Delta t\leq T.
	\end{cases}
	\end{align}
\end{lemma}
\begin{proof}
	 We first consider the case when $ i\in I $. On applying the operator $ L_{\varepsilon,cen}^{N,M} $ on the barrier function $ \phi_{i}^{n}(\mu) $, we get 
	\begin{align}
	\label{3.11}
	L_{\varepsilon,cen}^{N,M}\phi_i^n(\mu)=r_i^{-}\phi_{i-1}^n(\mu) + r_i^0\phi_{i}^n(\mu) + r_i^+\phi_{i+1}^n(\mu),
	\end{align}
	where
	\begin{align*}
	r_i^- = \frac{2\varepsilon}{\widehat{h_i} h_{i}}-\frac{a_i^n}{ \widehat{h_i}}, \quad
	r_i^0 = \frac{-2\varepsilon}{h_i h_{i+1}}-c_i^n, \quad
	r_i^+ =\frac{2\varepsilon}{\widehat{h_i}h_{i+1}}+\frac{a_i^n}{\widehat{h_i}}.
	\end{align*}
	On Simplifying the Eqn. (\ref{3.11}) using (\ref{3.9}), we get
	\begin{align*}
	L_{\varepsilon,cen}^{N,M}\phi_i^n(\mu) &\leq -\left( \frac{\mu}{\sqrt{\varepsilon} + \mu h_{i+1}} \right) \left(-2 \sqrt{\varepsilon} \mu +a_i^{n} + c_i^{n} \left( \frac{\sqrt{\varepsilon}+ \mu h_{i+1}}{\mu} \right)  \right) \phi_i^n(\mu)\\
	&\leq \dfrac{-C}{\sqrt{\varepsilon}}\phi_i^n(\mu), \quad i\in I, \thinspace n\Delta t\leq T.
	\end{align*}
	Next, we consider the case when $ i\notin I $. On applying the operator $ L_{\varepsilon,mu}^{N,\Delta t} $ on the barrier function $ \phi_{i}^{n}(\mu) $, we get 
	\begin{align}
	\label{3.12}
	L_{\varepsilon,mu}^{N,M}\phi_i^n(\mu)=s_i^-\phi_{i-1}^n(\mu) + s_i^0\phi_{i}^n(\mu)+s_i^+\phi_{i+1}^n(\mu),
	\end{align}
	where
	\begin{align*}
	s_i^- = \frac{2\varepsilon}{\widehat{h_i} h_i},\quad
	s_i^0 = \frac{-2\varepsilon}{h_ih_{i+1}}-\frac{a_{i+1/2}^n}{h_{i+1}}-\frac{c_{i+1/2}}{2},\quad
	s_i^+=\frac{2\varepsilon}{\widehat{h_i}h_{i+1}}+\frac{a_{i+1/2}^n}{{h_{i+1}}}-\frac{c_{i+1/2}}{2}.
	\end{align*}
	On simplifying Eqn. (\ref{3.12}) using (\ref{3.9}), we get
	\begin{align*}
	L_{\varepsilon,mu}^{N,M}\phi_i^n (\mu) &\leq - \frac{\mu}{\sqrt{\varepsilon}+ \mu h_{i+1}} \left( a_{i+1/2}^n + \frac{\sqrt{\varepsilon} c_{i+1/2}}{2 \mu} + \frac{c_{i+1/2}}{2} \left( \frac{\sqrt{\varepsilon}+\mu h_{i+1}}{\mu} \right) \right)\phi_i^n(\mu) \\
	&\leq \dfrac{-C}{\sqrt{\varepsilon}+\mu h_{i+1}}\phi_i^n(\mu), \quad \forall \thinspace i\notin I,\thinspace n\Delta t\leq T.  
	\end{align*}
\end{proof}
 \begin{lemma}
	\label{lm3.5}
	For each $ 0 \leq i \leq N$ and $ 0<\mu<\dfrac{m}{2} $, we have the following inequalities
	\begin{enumerate}
		\item[(i)] $ \exp\left(\dfrac{-m x_i}{\sqrt{\varepsilon}}\right)\leq \phi_i^n(\mu) , \quad 0 \leq i \leq N, \quad n \Delta t \leq T; $
		\item[(ii)] $\phi_i^n \leq \begin{cases}
		C L^{\displaystyle 2\mu\sigma_0 \left(\frac{i}{N}\right)} N^{\displaystyle -2\mu\sigma_0 \left( \frac{i}{N} \right)} ,\quad 1 \leq i \leq \frac{N}{2},\quad n\Delta t\leq T,\\
		C L^{ \displaystyle \mu\sigma_0} N^{ \displaystyle -\mu\sigma_0} ,\quad \frac{N}{2} \leq i \leq N, \quad n\Delta t \leq T.
		\end{cases}$
	\end{enumerate}
\end{lemma}
 \begin{proof}
	\begin{enumerate}
		\item[(i)] Using $e^{-x} = (1+x)^{-1},~ x \geq 0,$ we obtain the desired inequality.
		\item[(ii)]Considering the barrier function $\phi_i^n(\mu)$ for $ i \in \{1,2,\ldots,\frac{N}{2}\}$, we have
		\begin{align*}
		\phi_i^n(\mu)=\prod\limits_{j=1}^i\left(1+\frac{\mu h_j}{\sqrt{\varepsilon}}\right)^{-1}
		&=\left(1+\frac{\mu h}{\sqrt{\varepsilon}}\right)^{-i}\\
		&\leq \exp\left(\frac{-\mu x_i}{\sqrt{\varepsilon}+\mu h}\right) =\exp\left(\frac{-\mu i h }{\sqrt{\varepsilon}+\mu h}\right)\\
		&=\exp\left(\frac{-\mu i 2\sigma_0 L N^{-1}\sqrt{\varepsilon}}{\sqrt{\varepsilon}+\mu \sqrt{\varepsilon}\sigma_0 2N^{-1}L}\right).
		\end{align*}
		Using $ e^{-L} \leq L/N $, we get
		\begin{align*}
		\phi_i^n(\mu)&= L^{\displaystyle \frac{2\mu \sigma_0(\frac{i}{N})}{1+2\sigma_0\mu N^{-1} \ln{N}}} N^{\dfrac{-2\mu \sigma_0(\frac{i}{N})}{1+2\sigma_0\mu N^{-1} \ln{N}}} \leq C L^{ \displaystyle 2\mu\sigma_0 \left( \frac{i}{N} \right)} N^{\displaystyle -2\mu\sigma_0 \left( \frac{i}{N} \right)}.
		\end{align*}
		Next, consider the barrier function $\phi_i^n(\mu)$ for $ i \in \{\frac{N}{2},\ldots,N\}$, we have
		\begin{align*}
		\phi_i^n(\mu)&=\prod\limits_{j=1}^i\left(1+\frac{\mu h_j}{\sqrt{\varepsilon}}\right)^{-1}\\
		&\leq \prod\limits_{j=1}^{N/2}\left(1+\frac{\mu h_j}{\sqrt{\varepsilon}}\right)^{-1}\leq \exp\left(\frac{-\mu x_{N/2}}{\sqrt{\varepsilon}+\mu h}\right)\\\
		&=\exp\left(\frac{-\mu\sigma}{\sqrt{\varepsilon}+2\mu\sigma N^{-1}}\right)=\exp\left(\frac{-\mu\sigma_0 L\sqrt{\varepsilon}}{\sqrt{\varepsilon}+2\mu\sigma_0L\sqrt{\varepsilon} N^{-1}}\right)\\
		&=\exp\left(\frac{-\mu\sigma_0L}{1+2\mu\sigma_0L N^{-1}}\right).
		\end{align*}
		Using $ e^{-L} \leq L/N $, we get
		\begin{align*}
		\phi_i^n(\mu)&=L^{\left(\dfrac{\mu\sigma_0}{1+2\mu\sigma_0L N^{-1}}\right)} N^{\left(\dfrac{-\mu\sigma_0}{1+2\mu\sigma_0L N^{-1}}\right)}\\
		&\leq C L^{ \displaystyle \mu \sigma_0} N^{\displaystyle -\mu\sigma_0}.
		\end{align*}
	\end{enumerate}
\end{proof}
To obtain $\varepsilon$-uniform error estimate we decompose the numerical solution $ U_i^{n}=U(x_i,t_n) $ of the discrete problem (\ref{3.1}) into a regular part $ Y_i^{n} $ and a singular part $ Z_i^{n}$ analogously to the decomposition of the continuous solution $u(x,t)$ as:
\begin{align}
\label{3.13}
U(x_i,t_n)=Y(x_i,t_n)+Z(x_i,t_n), \quad \forall \thinspace (x_i,t_n)\in Q_{\sigma}^{N,M},
\end{align}
where $ Y_i^{n} $ satisfies the following non-homogeneous problem
\begin{align*}
L_\varepsilon^{N,M}Y(x_i,t_n)&=e(x_i,t_n)Y(x_i,t_{n-m_\tau})+f(x_i,t_n) ,\quad \forall \thinspace (x_i,t_n)\in Q_{\sigma}^{N,M},\\
Y(x_i,t_n)&=y(x_i,t_n), \quad \forall \thinspace (x_i,t_n)\in \Gamma_{\sigma}^{N,M}  
\end{align*}
and $ Z_i^{n} $ satisfies the following homogeneous problem 
\begin{align*}
L_\varepsilon^{N,M}Z(x_i,t_n)&=e(x_i,t_n)Z(x_i,t_{n-m_\tau}),\quad \forall \thinspace (x_i,t_n)\in Q_{\sigma}^{N,M},\\
Z(x_i,t_n)&=z(x_i,t_n), \quad \forall \thinspace (x_i,t_n)\in\Gamma_{\sigma}^{N,M}. 
\end{align*}
As a result the pointwise error at the node $ (x_i,t_n) $ in the discrete solution can be decomposed as
\begin{align*}
(U-u)(x_i,t_n)=(Y-y)(x_i,t_n)+(Z-z)(x_i,t_n),\quad \forall\thinspace(x_i,t_n)\in Q_{\sigma}^{N,M}.
\end{align*}

 \begin{lemma}[Error in the Regular Component ]
	\label{lm 3.6}
	Under the assumption (\ref{3.7}) of Lemma {\ref{3.1}}, the regular component at each mesh points  $ (x_i,t_n)\in \overline{Q}^{N,M}_{\sigma} $, satisfies the following error estimate: 
	\begin{align*}
	|(Y-y)(x_i,t_n)|\leq
	C(\Delta t+ N^{-2}), \quad  0 \leq i \leq N,\thinspace n\Delta t \leq T.
	\end{align*}
\end{lemma}
\begin{proof}  On the interval $[0, \tau]$, the right-hand side of (\ref{1.1}) becomes $f (x, t)+e(x, t)s(x,t-\tau)$ which is known and is independent of $\varepsilon$. We will consider two cases depending upon the relation between $\varepsilon$ and $N$:
\begin{enumerate}
\item[\textbf{Case (i)}] When $\varepsilon > \|a\|_{\overline{Q}}/N$. In this case we have $a_i h_i < 2 \varepsilon$ for all $i \in \{1,\ldots,N-1\} $ which implies the set $\{1, \ldots, N-1\} \subseteq I$. We get
 \begin{align*}
 |L_{\varepsilon}^{N,M}(Y-y)(x_i,t_n)| &\leq |L_{\varepsilon,cen}^{N,M}(Y-y)(x_i,t_n)|,\quad \forall  ~1\leq i \leq N-1,\thinspace n\Delta t \leq \tau  \\
& \leq C[\Delta t+h_i(h_{i+1}+h_i)(\varepsilon|y_{xxxx}|+|y_{xxx}|)].
\end{align*} 
Using $ h_{i+1}+h_i\leq 2N^{-1} $ and the bounds on the derivatives of $ y $ given in Theorem \ref{thm1}, we get
  \begin{align*}
 |L_{\varepsilon}^{N,M}(Y-y)(x_i,t_n)|\leq C(\Delta t+ N^{-2}),\quad 1 \leq i \leq N-1,\thinspace n\Delta t \leq \tau.
 \end{align*}  
 Using Lemma \ref{lm3.3} we can obtain the desired result.
\item[\textbf{Case (ii)}] When $\varepsilon \leq \|a\|_{\overline{Q}}/N$. For the smooth component the truncation error is defined as
	\begin{align*}
	|L_{\varepsilon}^{N,M}(Y-y)(x_i,t_n)|\leq
	\begin{cases}
	C[\Delta t+h_i(h_{i+1}+h_i)(\varepsilon|y_{xxxx}|+|y_{xxx}|)],\quad \forall \thinspace i\in I,\\
	C[\Delta t+\varepsilon(h_{i+1}+h_i)|y_{xxx}|+h_{i+1}^2(|y_{xxx}|+|y_{xx}|+|y_x|)],\quad \forall \thinspace i\not\in I.\end{cases}
	\end{align*}
	Using $ h_{i+1}+h_i\leq 2N^{-1} $ and the bounds on the derivatives of $ y(x,t) $ given in Theorem \ref{thm1}, we get
	\begin{align*}
	|L_{\varepsilon}^{N,M}(Y-y)(x_i,t_n)|\leq \begin{cases}
	C(\Delta t+ N^{-2}),\quad \forall~ i\in I,\\
	C(\Delta t+N^{-1}(\varepsilon+N^{-1})), \quad \forall~ i\not\in I. \end{cases}
	\end{align*}
	Using $\sqrt{\varepsilon} \leq C/N$, we get
	\begin{align*}
	|L_{\varepsilon}^{N,M}(Y-y)(x_i,t_n)|\leq 
	C(\Delta t+ N^{-2}),
	\end{align*}
	where $(x_i,t_n) \in Q_{\sigma,~ \tau}^{N,M}= \Omega_{\sigma}^{N} \times \Omega_{1}^{M}$ ($\Omega_{1}^{M}$ is the uniform mesh with $M=m_\tau$ mesh elements in the interval $[0,\tau]$). Applying Lemma \ref{lm3.3} we can obtain the desired result for the considered case. 
	\end{enumerate}
	Combining both the cases, we get
	\begin{align}
	\label{3.14}
	|(Y-y)(x_i,t_n)|\leq
	C(\Delta t+ N^{-2}),~ \forall (x_i,t_n) \in Q_{\sigma,\tau}^{N,M}.
	\end{align}
	The regular part $Y_i^n=Y(x_i,t_n)$ of the numerical solution on $Q_{\sigma,~ \tau}^{N,M}$ is denoted by $Y_{\tau,i}^n=Y_{\tau}(x_i,t_n)$.\\
	Next, we consider the second interval $[\tau, 2 \tau]$. On the second interval $[\tau, 2 \tau]$, the delay term $u(x,t-\tau)$ is the numerical solution obtained in the first interval $[0, \tau]$. To do the error analysis in the interval $[\tau,2 \tau]$ consider the problem:-
	\begin{align*}
	L_{\varepsilon}y(x,t)&= \left( \varepsilon \frac{\partial^2{y}}{\partial{x^2}}+a \frac{\partial{y}}{\partial{x}}-b\frac{\partial{y}}{\partial{t}}-c y\right) (x,t)= e(x,t) y(x, t- \tau)+f(x,t),\\ ~ for ~ (x,t) \in Q_2 = (0,1) \times (\tau,2 \tau],\\
	y(x,t) &= u(x,t) ~ for ~ (x,t) \in \Omega \times [0,\tau], \\
	y(0,t) &= q_{0}(t)\hspace{0.3cm} for \hspace{0.3cm}  \tau \leq t \leq  2 \tau,\nonumber \\
	y(1,t) &= q_{1}(t)\hspace{0.3cm} for \hspace{0.3cm} \tau \leq t\leq 2 \tau.	
	\end{align*}
	The regular component $Y(x_i,t_n)$ of the numerical solution at mesh point $(x_i,t_n) \in Q_{\sigma,~ 2\tau}^{N,M}= \Omega_{\sigma}^{N} \times \Omega_{2}^{M}$ ($\Omega_{2}^{M}$ is the uniform mesh with $M=m_\tau$ mesh elements in the interval $[\tau,2\tau]$) is determined by
	\begin{align*}
L_{\varepsilon}^{N,M}Y(x_i,t_n) = e(x_i,t_n) Y(x_i,t_{n-m_\tau}) + f(x_i,t_n), \quad (x_i,t_n) \in Q_{\sigma,~ 2\tau}^{N},\\
Y(0,t_n)= y(0,t_n), \quad Y(1,t_n)= y(1,t_n), \quad for ~ t_n \in \Omega_{2}^{M},\\
Y(x_i,t_n) = Y_{\tau}(x_i,t_n), \quad (x_i,t_n) \in Q_{\sigma,~ \tau}^{N,M}.	
	\end{align*} 
	We observe that 
	\begin{align*}
	L_{\varepsilon}^{N,M}(Y-y)(x_i,t_n) &= e(x_i,t_n) Y(x_i,t_{n-m_\tau}) + f(x_i,t_n) - L_{\varepsilon}^{N,M} y(x_i,t_n)\\
	&= e_{i}^n (Y_{\tau}-y)(x_i,t_n) + L_{\varepsilon}y(x_i,t_n) - L_{\varepsilon}^{N,M}y(x_i,t_n)\\
	&= e_i^n (Y_{\tau}-y)(x_i,t_n) + (L_{\varepsilon} - L_{\varepsilon}^{N,M})y(x_i,t_n).
	\end{align*}
	On using inequality (\ref{3.14}) and again considering two cases when $(i)\varepsilon > \|a\|_{\overline{Q}}/N$ and when $(ii)\varepsilon \leq \|a\|_{\overline{Q}}/N$, we get
	\begin{align*}
L_{\varepsilon}^{N,M}(Y-y)(x_i,t_n) = C[\Delta t + N^{-2}], \quad (x_i,t_n) \in Q_{\sigma,2 \tau}^{N,M}.
	\end{align*}
	Application of Lemma \ref{lm3.3} gives us the desired result in the interval $[\tau,2\tau]$. Similarly, we can prove the result for $t \in [2\tau,3 \tau] $ and so on.  
\end{proof}
\begin{lemma}[Error in the Singular Component]
	\label{lm3.7}
	Under the assumption (\ref{3.7}) and $ 0<\mu<\frac{m}{2} $ the following error estimate is satisfied by the singular component $Z(x_i,t_n)$ at each mesh points $(x_i,t_n)\in \overline{Q}^{N,M}_{\sigma}$,
	\begin{align*}
	|(Z-z)(x_i,t_n)|\leq 
	C(\Delta t+N^{-2}L^2 ), \quad  0 \leq i \leq N,\thinspace n\Delta t \leq T.
	\end{align*}
\end{lemma}
\begin{proof}
	 We proceed by method of steps. Firstly, we compute the error in the interval $[0,\tau]$ and then consider the interval $[\tau,2 \tau]$.\\
	 
	Case (i): In this case, for $t \in [0,\tau]$ the right hand side of Eqn. (\ref{1.1}) is known and independent of $\varepsilon$. We first consider the outer region $ (\sigma,1]\times (0,\tau] $. We know that $Z$ and $z$ are small in the outer region irrespective of the fact that whether $i \in I$ or $i \notin I$.\\
	Consider the barrier functions
	\begin{align*}
	\psi^\pm(x_i,t_n)=C\phi_i^n(\mu)\pm Z(x_i,t_n),\quad \forall \thinspace (x_i,t_n)\in Q^{N,M}_{\sigma,\tau},
	\end{align*}
	where $ C=|z(x_0,t_n)| $.
	We observe that 
	\begin{align*}
	\psi^\pm(x_0,t_n)&=C\phi_0^n(\mu)\pm Z(x_0,t_n)=C\pm z(x_0,t_n)\geq 0,\\
	\psi^\pm(x_i,t_n)&=C\phi_i^n(\mu)\geq 0 ,\quad \forall ~ (x_i,t_n) \in \Gamma_{b,\sigma}^{N} \quad
	\mbox{and}\\
	\psi^\pm(x_N,t_n)&=C\phi_N^n(\mu)\geq 0;\\
	L_\varepsilon^{N,M}\psi^\pm(x_i,t_n)&=C L_\varepsilon^{N,M}\phi_i^n(\mu)\pm L_\varepsilon^{N,M} Z(x_i,t_n)= L_\varepsilon^{N,M}\phi_i^n(\mu)< 0.
	\end{align*}
	Using the discrete minimum principle, we get
	\begin{align*}
	|Z(x_i,t_n)|&\leq C\phi_i^n(\mu)=C\prod\limits_{j=1}^i \left(1+\frac{\mu h_j}{\sqrt{\varepsilon}}\right)^{-1} ,\quad \forall \thinspace (x_i,t_n)\in Q^{N,M}_{\sigma,\tau}.
	\end{align*}
	Using triangle inequality and Lemma \ref{lm3.5}, we get
	\begin{align*}
	|(Z-z)(x_i,t_n)|&\leq |Z(x_i,t_n)|+|z(x_i,t_n)|\\
	&\leq C\prod\limits_{j=1}^i \left(1+\frac{\mu h_j}{\sqrt{\varepsilon}}\right)^{-1}+C\exp\left(\frac{-m x_i}{\sqrt{\varepsilon}}\right)\\
	&\leq C\prod\limits_{j=1}^i \left(1+\frac{\mu h_j}{\sqrt{\varepsilon}}\right)^{-1}=C \phi_i^n(\mu).
	\end{align*}
	Using the bounds of $ \phi_i^n(\mu) $ given in Lemma \ref{lm3.5}, we get
	\begin{align*}
	|(Z-z)(x_i,t_n)|\leq C L^{\mu \sigma_0} N^{-\mu \sigma_0},\quad \forall ~ \frac{N}{2}\leq i
	\leq N,~ t_n\in \Omega_{1}^{M}.
	\end{align*}
	On taking $ \sigma_0=\dfrac{2}{\mu} $, we get
	\begin{align}
	\label{3.19}
	|(Z-z)(x_i,t_n)|\leq C L^{2} N^{-2},\quad \forall ~ \frac{N}{2}\leq i \leq N, ~ t_n\in \Omega_{1}^{M}.
	\end{align}
	
	Next, we consider the inner region $[0,\sigma] \times (0,\tau]$. For $N \geq N_0$, satisfying (\ref{3.9}) we have $a_i^{n}h_i < 2 {\varepsilon}$, for all $i=1,\ldots,N/2$ which implies $\lbrace 1,\ldots,N/2\rbrace \subseteq I$. Therefore, $ L_{\varepsilon,cen}^{N,M} $ is applied in the region $[0,\sigma] \times [0,\tau]$. We have,
	\begin{align} \label{3.19a}
	|L_{\varepsilon,cen}^{N,M}(Z-z)(x_i,t_n)|&\leq C \left[\Delta t+h_i\int\limits_{x_{i-1}}^{x_{i+1}} \left(\varepsilon|z_{xxxx}|+|z_{xxx}| \right) dx \right] \nonumber \\
	&\leq C \left[\Delta t+\frac{h_i}{\varepsilon^{3/2}}\int\limits_{x_{i-1}}^{x_{i+1}}\exp \left(\frac{-m x}{\sqrt{\varepsilon}} \right)dx \right] \nonumber \\
	&=C \left[\Delta t+\frac{h_i}{m {\varepsilon}}\left\lbrace\exp \left(\frac{-m x_{i-1}}{\sqrt{\varepsilon}} \right)-\exp \left(\frac{-m x_{i+1}}{\sqrt{\varepsilon}} \right)\right\rbrace \right] \nonumber \\
	&=C \left[\Delta t+\frac{h_i}{m{\varepsilon}}\exp \left(\frac{-m x_{i}}{\sqrt{\varepsilon}} \right)\left\lbrace\exp \left(\frac{m h}{\sqrt{\varepsilon}} \right)-\exp \left(\frac{- m h}{\sqrt{\varepsilon}} \right)\right\rbrace \right] \nonumber \\ 
	&=C \left[\Delta t+\frac{h_i}{m{\varepsilon}} \exp \left( \frac{-m x_{i}}{\sqrt{\varepsilon}} \right) \text{sinh} \left( \frac{m h}{\sqrt{\varepsilon}} \right) \right].  
	\end{align}
	Since for inner region $i \in I$ so, using (\ref{3.9}) we get $ m h <  \sqrt{\varepsilon} $ and sin$h\xi \leq C\xi$, for $ 0 \leq \xi \leq 2 $. This implies sinh$ \left( \frac{m h}{\sqrt{\varepsilon}} \right) \leq \frac{C m h}{\sqrt{\varepsilon}} $. Therefore, (\ref{3.19a}) becomes
	\begin{align}
	\label{3.20}
	| L_{\varepsilon,cen}^{N,M} (Z-z)(x_i,t_n) |  \leq C \left[ \Delta t + \frac{h_i^{2}}{\varepsilon^{3/2}} \exp \left( \frac{-m x_i}{\sqrt{\varepsilon}} \right) \right].
	\end{align} 
	Also, in the inner region, we have $h_i=h$, so
	\begin{align}
	\label{3.23}
	|L_{\varepsilon,cen}^{N,M}(Z-z)(x_i,t_n)|&\leq C \left[\Delta t+\frac{h^2}{\varepsilon^{3/2}} \exp \left( \frac{-m x_i}{\sqrt{\varepsilon}} \right) \right]
	\leq C \left[ \Delta t +\frac{ N^{-2}L^{2}}{\sqrt{\varepsilon}} \phi_i^{n}(\mu) \right] .
	\end{align}  
	We have $| (Z-z)(x_0,t_n) | = 0$, $\forall ~ t_n \in \Omega_{1}^{M}$ and $ |(Z-z)(x_i,t_n)| = 0 $, $\forall ~ (x_i,t_n) \in \Gamma_{b,\sigma}^{N}$.\newline
	Also, from (\ref{3.19}) we have 
	\begin{align*}
	|(Z-z)(x_{N/2},t_n)| \leq CN^{-2}L^{2}.
	\end{align*}
	Considering the barrier functions, 
	\begin{align*}
	\Psi^{\pm}(x_i,t_n) &= C (  N^{-2}L^2 \phi_i^{n}(\mu) + ( \Delta t + N^{-2}L^{2})t_n) \pm (Z-z)(x_i,t_n),
	\end{align*}
	we observe that $\Psi^{\pm}(x_{N/2},t_n)>0 $, $\Psi^{\pm}(x_N,t_n)>0$, $\forall ~ t_n \in \Omega_{1}^{M}$ and $\Psi^{\pm}(x_i,t_n)>0$, $\forall  ~ (x_i,t_n) \in \Gamma_{b,\sigma}^{N}$. Using Lemma \ref{lm3.4} we have,
	\begin{align*}
	L_{\varepsilon,cen}^{N,M}\Psi^{\pm}(x_i,t_n)&= C( N^{-2}L^2 L_{\varepsilon}^{N,M} \phi_i^{n}(\mu)) - b(x_i,t_n)(\Delta t + N^{-2}L^{2}) \pm L_{\varepsilon}(Z-z)(x_i,t_n) \leq 0  .
	\end{align*} 
	From discrete minimum principle we get, 
	\begin{align}
	\label{3.23a}
	|(Z-z)(x_i,t_n)| &\leq C(\Delta t + N^{-2}L^{2} \phi_i^{n}(\mu)), \qquad \forall~ i=0,\ldots,N/2, \quad t_n \in \Omega_{1}^{M},\nonumber \\
	&\leq C(\Delta t + N^{-2}L^{2}).
	\end{align}
	 The singular part $Z(x_i,t_n)$ of the numerical solution on $Q_{\sigma,~ \tau}^{N,M}$ is denoted by $Z_{\tau}(x_i,t_n)$.\\
	 
	Case (ii): On the second interval $[\tau, 2 \tau]$, the delay term $u(x,t-\tau)$ is the numerical solution obtained in the first interval $[0, \tau]$. We will do the error analysis over the interval $[\tau,2 \tau]$ in the following way. We will consider the singularly perturbed delay parabolic partial differential equation (\ref{2.8}) on the second interval $[\tau,2 \tau]$. The singular component $Z(x_i,t_n)$ of the numerical solution at mesh points $(x_i,t_n) \in Q_{\sigma,~ 2\tau}^{N,M}= \Omega_{\sigma}^{N} \times \Omega_{2}^{M}$ is determined by
	\begin{align*}
	L_{\varepsilon}^{N,M} Z = e(x_i,t_n) Z(x_i,t_{n-m_\tau}), \quad (x_i,t_n) \in Q_{\sigma,~ 2\tau}^{N},\\
	Z(x_i,t_n)=z(x_i,t_n), \quad (x_i,t_n) \in \Gamma_{\sigma}^{N,M}.	
	\end{align*} 
	We observe that 
	\begin{align*}
	L_{\varepsilon}^{N,M}(Z-z)(x_i,t_n) &= e(x_i,t_n) Z(x_i,t_{n-m_\tau}) - L_{\varepsilon}^{N,M} z(x_,t_n)\\
	&= e_{i}^n (Z_{\tau}-z)(x_i,t_{n-m_\tau}) + L_{\varepsilon}z(x_i,t_n) - L_{\varepsilon}^{N,M}z(x_i,t_n)\\
	&= e_i^n (Z_{\tau}-z)(x_i,t_{n-m_\tau}) + (L_{\varepsilon} - L_{\varepsilon}^{N,M})z(x_i,t_n)
	\end{align*}
    First term on the right hand side can be approximated using (\ref{3.23a}) and the second term can be approximated using the same approach as discussed in the first interval. The proof is completed by introducing the barrier functions and applying the discrete minimum principle as in the previous case for $t \in [0,\tau]$.
    \begin{align*}
    |(Z-z)(x_i,t_n)|	& \leq C[\Delta t + N^{-2}L^{2}],\quad (x_i,t_n) \in \overline{Q}_{\sigma,2\tau}^{N,M}.
	\end{align*}
	 Similarly, the case, for $t \geq 2\tau $ also follows on the same lines.  
\end{proof}
\begin{theorem}
\label{thm3.1}
Let $u(x_i,t_n)$ be the exact solution of the problem (\ref{1.1})-(\ref{1.3}) and $U(x_i,t_n)$ be the discrete solution of the system of Eqns.(\ref{3.1}) at each mesh point $(x_i,t_n) \in \overline{Q}_{\sigma}^{N,M}$. Then under the assumption of Lemma \ref{lm3.4} and $ 0 < \mu < m/2$ for $ N \geq N_0 $ we have
\begin{align*}
\|(U-u)\|_{\overline{Q}_{\sigma}^{N,M}} \leq 
C(\Delta t + N^{-2}L^{2} + N^{-2}).
\end{align*} 
\end{theorem}
\section{Richardson extrapolation}
In this section, the Richardson extrapolation technique is used to obtain higher accuracy and order of convergence in time direction. We consider two meshes $ \overline{Q}_{\sigma}^{N,M} = \overline{\Omega}_{\sigma}^{N} \times \overline{\Omega}^{M}$ and $ \overline{Q}_{\sigma}^{N,2M} = \overline{\Omega}_{\sigma}^{N} \times \overline{\Omega}^{2M}$ where $\overline{\Omega}^{M}$ and $\overline{\Omega}^{2M}$ are uniform meshes with $M$ and $2M$ mesh points, respectively in the temporal direction. Both the considered meshes have same number of mesh points in the spatial direction. Let
\begin{align*}
\overline{Q}_{\sigma,0}^{N,M}= \overline{Q}_{\sigma}^{N,M} \cap \overline{Q}_{\sigma}^{N,2M}.
\end{align*}
Then, $\overline{Q}_{\sigma,0}^{N,M}= \overline{Q}_{\sigma}^{N,M}$, as $\overline{Q}_{\sigma}^{N,M} \subseteq \overline{Q}_{\sigma}^{N,2M}$. Let $U^{k}$ denote the numerical solution of the problem (\ref{3.1}) on the mesh $ \overline{Q}_{\sigma}^{N,kM} = \overline{\Omega}_{\sigma}^{N} \times \overline{\Omega}^{kM} $ where $k=1,2$. Then, we approximate $u(x,t)$ by $U_{ext}(x_i,t_n)$ where 
\begin{align}
	\label{4.1}
	U_{ext}(x_i,t_n)=(2U^{2}-U^{1})(x_i,t_n), \quad (x_i,t_n) \in \overline{Q}_{\sigma,0}^{N,M}.
\end{align}
The numerical approximation $U_{ext}$ has improved order of convergence in time. To verify this we use a technique similar to \cite{MR2201992}. We have
\begin{align}
	\label{4.2}
	U^{k}(x_i,t_n)=u(x_i,t_n)+2^{-(k-1)}\Delta t \xi^{k} (x_i,t_n) + R_{n}^{k}(x_i,t_n),\quad (x_i,t_n) \in \overline{Q}_{\sigma,0}^{N,M}
\end{align}
where $R_{n}^{k}$, $k=1,2$ is the remainder term and $ \xi^k $ is the solution of the following problem: 
\begin{align}
	\label{4.3}
	L_{\varepsilon} \xi^{k} &= \left(\frac{b}{2}\frac{\partial^2 u}{\partial t^2} \right) (x,t) , \quad (x,t) \in Q,\\ \nonumber
	\xi^{k}(x,t)&=0, \quad (x,t)\in  \Gamma .
\end{align}
We need to derive the estimates for the remainder term $R^{k}$ on $\overline{Q}_{\sigma}^{N,kM}$, $k=1,2$. Also $R^k(x_i,t_n)=0$, $\forall (x_i,t_n) \in \Gamma_{\sigma}^{N,kM}$, where $\Gamma_{\sigma}^{N,kM}$, $k=1,2$ is the boundary of $\overline{Q}^{N,kM}$. We have
\begin{align*}
|L^{N,k M}_{\varepsilon}R_n^{k}(x_i,t_n)|&= |L^{N,kM}_{\varepsilon}(U^{k}-u)(x_i,t_n)- 2^{-(k-1)} \Delta t L^{N, kM}_{\varepsilon}\xi^{k}(x_i,t_n)|\\
&\leq  C 
\left( N^{-2}L^{2}+N^{-2} + \Delta t^2 \right).
\end{align*}
From discrete minimum principle, we get
\begin{align*}
|R_n^{k}(x_i,t_n)| \leq
C \left( N^{-2}L^{2}+N^{-2} + \Delta t^2 \right).
\end{align*} 

\begin{theorem}
	\label{thm4.1}
	Let $u(x_i,t_n)$ be the exact solution of the problem.~(\ref{1.1})-(\ref{1.3}) and $U_{ext}(x_i,t_n)$ be the discrete solution obtained using the Richardson extrapolation at each mesh point $(x_i,t_n) \in \overline{Q}_{\sigma}^{N,M}$. Then, for $ N \geq N_0 $ where $N_0$ satisfies the assumption (\ref{3.7}) and $ 0 < \mu < m/2$, we have the following $ \varepsilon$-uniform error estimate
	\begin{align*}
	\|(U_{ext}-u)\|_{\overline{Q}_{\sigma}^{N,M}} \leq 
	C(\Delta t^2 + N^{-2}L^{2} + N^{-2}).
	\end{align*} 
\end{theorem} 

\section{Numerical results}
In this section, we present the numerical results for two test problems to validate the theoretical results. They also verify the high accuracy and convergence rate of the proposed schemes.
\begin{problem}
	Consider the following singularly perturbed parabolic IBVP :-
	\begin{equation}
	\begin{cases}
	
	\left( \varepsilon \frac{\partial^2{u}}{\partial{x^2}}+x^p \frac{\partial{u}}{\partial{x}}-\frac{\partial{u}}{\partial{t}}-u\right)(x,t) = (0.5)u(x,t-1) + x^2-1, \hspace{0.5cm} \forall ~ (x,t)\in Q = \Omega \times (0,2],\\
	u(x,t)=(1-x)^2,\hspace{0.5cm} \forall ~ (x,t)\in [0,1] \times [-1,0],\\
	u(0,t)=1+t^2, \quad u(1,t)=0,\quad \forall ~t\in (0,2].
	
	\end{cases}
	\end{equation}
\end{problem}
\begin{problem}
	Consider the following singularly perturbed parabolic IBVP :-
	\begin{equation}
	\begin{cases}
	
	\left( \varepsilon \frac{\partial^2{u}}{\partial{x^2}}+x^p \frac{\partial{u}}{\partial{x}}-\frac{\partial{u}}{\partial{t}}-(x+p)u\right)(x,t) =-u(x,t-1) + p \exp (-t)(x^2-1),\\ \hspace{0.5cm} \forall ~ (x,t)\in Q = \Omega \times (0,2],\\
	u(x,t)=(1-x)^2,\hspace{0.5cm} \forall ~ (x,t)\in [0,1] \times [-1,0],\\
	u(0,t)=1+t^2, \quad u(1,t)=0,\quad \forall ~t\in (0,2].
	
	\end{cases}
	\end{equation}
\end{problem}
Since exact solutions of the given problems are not known, the performances of the proposed schemes are illustrated by using the double mesh principle to calculate the maximum pointwise error. The maximum pointwise error is defined as
\begin{align*}
E^{N,M}_\varepsilon = \| \widetilde{U}^{N,M}(x_i,t_j)- \widetilde{U}^{2N,2M}(x_i,t_j)\|_{\overline{Q}^{N,M}_{\sigma}},
\end{align*}
where $\widetilde{U}=U$ for hybrid scheme (\ref{3.1}) and $\widetilde{U}=U_{ext}$ if Richardson extrapolation is applied on the scheme (\ref{3.1}). The corresponding order of convergence $q_{\varepsilon}^{N,M}$ is computed as
\begin{align*}
q^{N,M}_\varepsilon = \frac{\ln\left( {{E^{N,M}_\varepsilon}/E^{2N,2M}_\varepsilon}\right)}{\ln{2}}.
\end{align*}
Also, the $ \varepsilon $-uniform maximum pointwise error $E^{N,M}$ is computed as
\begin{align*}
E^{N,M} = \max\limits_{\varepsilon}E^{N,M}_\varepsilon
\end{align*}
and the corresponding $ \varepsilon $-uniform order of convergence $q^{N,M}$ is given by
\begin{align*}
q^{N,M} = \frac{\ln{\left( {E^{N,M}}/E^{2N,2M}\right)}}{\ln{2}}.
\end{align*}
For various values of $\varepsilon$, $N$ and $M$ the computed maximum pointwise errors $E_{\varepsilon}^{N,M}$ and the corresponding order of convergence $q_{\varepsilon}^{N,M}$ for the considered problems are tabulated in Tables \ref{tb1} to \ref{tb6}.\\
 In Table \ref{tb1} and \ref{tb3} we  have given the results for upwind scheme on uniform mesh, upwind scheme on piecewise uniform Shishkin mesh and hybrid scheme (\ref{3.1}) on modified Shishkin mesh for Problem 1, 2, respectively. It can be seen that the uniform mesh do not work. The upwind scheme on Shishkin mesh has almost first order of convergence. The proposed hybrid scheme gives better result than the upwind scheme with Shishkin mesh. The numerical results computed using hybrid scheme show monotonically decreasing behaviour as $N$ increases which confirms the $\varepsilon$-uniform convergence of the hybrid scheme (\ref{3.1}). The order of convergence of hybrid scheme is not depicting the theoretical order of convergence of order two upto a logarithmic factor  as proved in Theorem \ref{thm3.1} as the error consists of two parts due to spatial and temporal discretization. The hybrid scheme improves the accuracy space but temporal order of convergence remains one. As a result, the numerical results display almost first order of convergence due to the much influence of the temporal error. Therefore, Richardson extrapolation is used to increase the order of convergence in time.\\
Tables \ref{tb2} and \ref{tb4} show that the use of Richardson extrapolation on hybrid scheme further improves the error and the rate of convergence. The resulting scheme provides almost second order of convergence in space and second order of convergence in time variable.\\
Tables \ref{tb5} and \ref{tb6} display the numerical results computed using Richardson extrapolation for Problem 1 and 2, respectively, with different values of $p$.
\section{Conclusions}
In this article, we proposed and analysed a higher order numerical scheme for the solution of singularly perturbed parabolic problems with time delay and degenerate coefficient. The proposed scheme is comprised of implicit Euler scheme for time discretization on uniform mesh and a hybrid scheme for space discretization on modified Shishkin mesh. Parameter uniform convergence of order one in time direction and order two upto a logarithmic factor in space direction is established for the proposed scheme (\ref{3.1}). Further, to improve the order of convergence in time direction Richardson extrapolation is employed in the time direction. The resulting scheme increase the order of convergence to two in time direction. The numerical experiments are presented to illustrate the theoretical results.  

\section*{Acknowledgements}
The first author wish to acknowledge UGC Non-NET Fellowships for financial support vide Ref.No. Sch./139/Non-NET/Ext-152/2018-19/67.

\begin{table}[h!]
\begin{center}
	\caption{The maximum pointwise errors $ E^{N,M}_\varepsilon $ and the corresponding order of convergence $q^{N,M}_\varepsilon $ for the Problem 1 using different schemes with $p=1$ and $M=N$.}
	\label{tb1}
	\resizebox{\textwidth}{!}{\begin{tabular}{ c  c  c  c  c  c  c c c}
		\hline
		\textbf{$ \varepsilon \downarrow $} & \textbf{Scheme}  & \textbf{$ N=32 $} & \textbf{$ N=64 $} & \textbf{$ N=128 $} & \textbf{$ N=256 $} & \textbf{$ N=512 $} & \textbf{$ N=1024 $} & \textbf{$ N=2048 $} \\
		\hline
		{$2^{-8}$} & \textbf{Upwind scheme on} & 7.061e-02 & 3.726e-02 & 1.964e-02 & 1.009e-02 & 5.122e-03 & 2.581e-03 & 1.296e-03\\ 
	    & \textbf{uniform mesh} & 9.224e-01 & 9.236e-01 & 9.617e-01 & 9.774e-01 & 9.887e-01 & 9.942e-01\\
	    \hline
	    & \textbf{Upwind scheme on} & 6.350e-02 & 3.726e-02 & 1.964e-02 & 1.009e-02 & 5.122e-03 & 2.581e-03 & 1.296e-03\\ 
	    & \textbf{Shishkin mesh} & 7.693e-01 & 9.236e-01 & 9.617e-01 & 9.774e-01 & 9.887e-01 & 9.942e-01\\
		\hline
		&\textbf{Hybrid scheme on}& 4.308e-03 & 2.056e-03 & 1.031e-03 & 5.177e-04 & 2.647e-04 & 1.352e-04 & 6.834e-05\\ 
		& \textbf{Shishkin mesh}& 1.068e+00 & 9.957e-01 & 9.936e-01 & 9.676e-01 & 9.694e-01 & 9.843e-01\\
		\hline
		{$2^{-12}$}   & \textbf{Upwind scheme} & 1.481e-01 & 1.390e-01 & 6.916e-02 & 3.640e-02 & 1.927e-02 & 9.902e-03 & 5.028e-03\\
		& \textbf{uniform mesh} & 9.131e-02 & 1.007e+00 & 9.259e-01 & 9.179e-01 & 9.604e-01 & 9.777e-01\\
		\hline
		& \textbf{Upwind scheme on} & 6.438e-02 & 3.929e-02 & 2.365e-02 & 1.385e-02 & 7.897e-03 & 4.424e-03 & 2.443e-03\\
		& \textbf{Shishkin mesh} & 7.125e-01 & 7.320e-01 & 7.720e-01 & 8.107e-01 & 8.360e-01 & 8.565e-01\\  
		\hline
		&\textbf{Hybrid scheme on}& 4.124e-03 & 2.097e-03 & 1.054e-03 & 5.295e-04 & 2.728e-04 & 1.395e-04 &	7.084e-05\\
		& \textbf{Shishkin mesh}& 9.760e-01 & 9.927e-01 & 9.925e-01 & 9.571e-01 & 9.670e-01 & 9.780e-01\\ 
		\hline
		{$2^{-16}$}  & \textbf{Upwind scheme} & 2.176e-02 & 6.828e-02 & 1.466e-01 & 1.379e-01 & 6.858e-02 & 3.610e-02 & 1.912e-02\\
		& \textbf{uniform mesh} & -1.649e+00 & -1.102e+00 & 8.809e-02 & 1.008e+00 & 9.258e-01 & 9.169e-01\\
		\hline
		& \textbf{ Upwind scheme on} & 6.444e-02 & 3.932e-02 & 2.368e-02 & 1.386e-02 & 7.905e-03 & 4.428e-03 & 2.445e-03\\
		& \textbf{Shishkin mesh} & 7.128e-01 & 7.318e-01 & 7.720e-01 & 8.107e-01 & 8.360e-01 & 8.566e-01\\
		\hline
		&\textbf{Hybrid scheme on}& 4.126e-03 & 2.099e-03 & 1.055e-03 & 5.303e-04 & 2.739e-04 & 1.402e-04 &	7.119e-05\\
		& \textbf{Shishkin mesh}& 9.752e-01 & 9.923e-01 & 9.924e-01 & 9.530e-01 & 9.664e-01 & 9.777e-01\\
		\hline
		{$2^{-20}$}  & \textbf{Upwind scheme} & 3.961e-03 & 6.240e-03 & 2.076e-02 & 6.750e-02 & 1.461e-01 & 1.376e-01 & 6.842e-02\\
		& \textbf{uniform mesh} & -6.558e-01 & -1.734e+00 & -1.701e+00 & -1.114e+00 & 8.623e-02 & 1.008e+00\\
		\hline
		& \textbf{ Upwind scheme on} & 6.444e-02 & 3.932e-02 & 2.368e-02 & 1.387e-02 & 7.905e-03 &	4.429e-03 & 2.446e-03\\
		& \textbf{Shishkin mesh} & 7.128e-01 & 7.317e-01 & 7.719e-01 & 8.106e-01 & 8.360e-01 & 8.566e-01\\
		\hline
		&\textbf{Hybrid scheme on}& 4.126e-03 & 2.099e-03 & 1.055e-03 & 5.304e-04 & 2.741e-04 & 1.403e-04 &
		7.124e-05\\
		& \textbf{Shishkin mesh}& 9.750e-01 & 9.922e-01 & 9.923e-01 & 9.523e-01 & 9.662e-01 & 9.776e-01\\
		\hline
		{$2^{-24}$}  & \textbf{Upwind scheme} & 3.961e-03 & 2.062e-03 & 1.843e-03 & 5.574e-03 & 2.037e-02 & 6.727e-02 & 1.460e-01\\
		& \textbf{uniform mesh} & 9.417e-01 & 1.618e-01 & -1.596e+00 & -1.870e+00 & -1.724e+00 & -1.118e+00\\
		\hline
		& \textbf{ Upwind scheme on} & 6.444e-02 & 3.932e-02 & 2.368e-02 & 1.387e-02 & 7.905e-03 &	4.429e-03 & 2.446e-03\\
		& \textbf{Shishkin mesh} & 7.128e-01 & 7.317e-01 & 7.719e-01 & 8.106e-01 & 8.360e-01 & 8.566e-01\\
		\hline
		&\textbf{Hybrid scheme on}& 4.126e-03 & 2.099e-03 & 1.055e-03 & 5.304e-04 & 2.741e-04 & 1.403e-04 &
		7.125e-05\\
		& \textbf{Shishkin mesh} &9.750e-01 & 9.922e-01 & 9.923e-01 & 9.521e-01 & 9.662e-01 & 9.776e-01\\
		\hline
		{$2^{-28}$}  & \textbf{Upwind scheme} & 3.961e-03 & 2.062e-03 & 1.057e-03 & 5.943e-04 & 1.476e-03 & 5.380e-03 & 2.027e-02\\
		& \textbf{uniform mesh} & 9.417e-01 & 9.635e-01 & 8.312e-01 & -1.312e+00 & -1.866e+00 & -1.913e+00\\
		\hline
		 & \textbf{ Upwind scheme on} & 6.444e-02 & 3.932e-02 & 2.368e-02 & 1.387e-02 & 7.906e-03 &	4.429e-03 & 2.446e-03\\
		& \textbf{Shishkin mesh} & 7.128e-01 & 7.317e-01 & 7.719e-01 & 8.106e-01 & 8.360e-01 & 8.566e-01\\
		\hline
		&\textbf{Hybrid scheme on}& 4.126e-03 & 2.099e-03 & 1.055e-03 & 5.304e-04 & 2.741e-04 & 1.403e-04 &
		7.126e-05\\
		& \textbf{Shishkin mesh}& 9.750e-01 & 9.922e-01 & 9.923e-01 & 9.521e-01 & 9.662e-01 & 9.776e-01\\
		\hline
		{$2^{-32}$}  & \textbf{Upwind scheme} & 3.961e-03 & 2.062e-03 & 1.057e-03 & 5.367e-04 & 2.702e-04 & 4.030e-04 & 1.378e-03\\
		& \textbf{uniform mesh} & 9.417e-01 & 9.635e-01 & 9.782e-01 & 9.900e-01 & -5.767e-01 & -1.774e+00\\
		\hline
		& \textbf{ Upwind scheme on} & 6.444e-02 & 3.932e-02 & 2.368e-02 & 1.387e-02 & 7.906e-03 &	4.429e-03 & 2.446e-03\\
		& \textbf{Shishkin mesh} & 7.128e-01 & 7.317e-01 & 7.719e-01 & 8.106e-01 & 8.360e-01 & 8.566e-01\\
		\hline
		&\textbf{Hybrid scheme on}& 4.126e-03 & 2.099e-03 & 1.055e-03 & 5.304e-04 & 2.741e-04 & 1.403e-04 &
		7.126e-05\\
		& \textbf{Shishkin mesh}& 9.750e-01 & 9.922e-01 & 9.923e-01 & 9.521e-01 & 9.662e-01 & 9.776e-01\\
		\hline
		{$2^{-36}$}  & \textbf{Upwind scheme} & 3.961e-03 & 2.062e-03 & 1.057e-03 & 5.367e-04 & 2.702e-04 & 1.356e-04 & 1.173e-04\\
		& \textbf{uniform mesh} & 9.417e-01 & 9.635e-01 & 9.782e-01 & 9.900e-01 & 9.953e-01 & 2.082e-01\\
		\hline
		& \textbf{ Upwind scheme on} & 6.444e-02 & 3.932e-02 & 2.368e-02 & 1.387e-02 & 7.906e-03 &	4.429e-03 & 2.446e-03\\
		& \textbf{Shishkin mesh} & 7.128e-01 & 7.317e-01 & 7.719e-01 & 8.106e-01 & 8.360e-01 & 8.566e-01\\
		\hline
		&\textbf{Hybrid scheme on}& 4.126e-03 & 2.099e-03 & 1.055e-03 & 5.304e-04 & 2.741e-04 & 1.403e-04 & 7.126e-05\\
		& \textbf{Shishkin mesh}& 9.750e-01 & 9.922e-01 & 9.923e-01 & 9.521e-01 & 9.662e-01 & 9.776e-01\\
		\hline
		{$2^{-40}$}  & \textbf{Upwind scheme} & 3.961e-03 & 2.062e-03 & 1.057e-03 & 5.367e-04 & 2.702e-04 & 1.356e-04 & 6.789e-05\\
		& \textbf{uniform mesh} & 9.417e-01 & 9.635e-01 & 9.782e-01 & 9.900e-01 & 9.953e-01 & 9.976e-01\\
		\hline
		& \textbf{ Upwind scheme on} & 6.444e-02 & 3.932e-02 & 2.368e-02 & 1.387e-02 & 7.906e-03 & 4.429e-03 & 2.446e-03\\
		& \textbf{Shishkin mesh} & 7.128e-01 & 7.317e-01 & 7.719e-01 & 8.106e-01 & 8.360e-01 & 8.566e-01\\
		\hline
		&\textbf{Hybrid scheme on}& 4.126e-03 & 2.099e-03 & 1.055e-03 & 5.304e-04 & 2.741e-04 & 1.403e-04 & 7.126e-05\\
		& \textbf{Shishkin mesh}& 9.750e-01 & 9.922e-01 & 9.923e-01 & 9.521e-01 & 9.662e-01 & 9.776e-01\\
		\hline
		{$ \mathbf{E^{N,M}}$} &\textbf{Upwind scheme on} & \textbf{1.481e-01} & \textbf{1.390e-01} & \textbf{1.466e-01} & \textbf{1.379e-01} & \textbf{1.461e-01} & \textbf{1.376e-01} & \textbf{1.460e-01} \\
	    & \textbf{uniform mesh} & \textbf{9.131e-02} & \textbf{-7.659e-02} & \textbf{8.809e-02} & \textbf{-8.307e-02} & \textbf{8.623e-02}  & \textbf{-8.492e-02} \\
	    \hline
	    {$ \mathbf{E^{N,M}}$}& \textbf{ Upwind scheme on} & \textbf{6.444e-02} & \textbf{3.932e-02} & \textbf{2.368e-02} & \textbf{1.387e-02} & \textbf{7.906e-03} & \textbf{4.429e-03} & \textbf{2.446e-03} \\
	    & \textbf{Shishkin mesh} & \textbf{7.128e-01} & \textbf{7.317e-01} & \textbf{7.719e-01} & \textbf{8.106e-01} & \textbf{8.360e-01}  & \textbf{8.566e-01}\\
	    \hline
		{$ \mathbf{E^{N,M}}$} &\textbf{ Hybrid scheme }& \textbf{4.308e-03} & \textbf{2.099e-03} & \textbf{1.055e-03} & \textbf{5.304e-03} & \textbf{2.741e-04} & \textbf{1.403e-04} & \textbf{7.126e-05} \\
		& \textbf{ on Shishkin mesh} (\ref{3.1}) & \textbf{1.0373} & \textbf{9.922e-01} & \textbf{9.923e-01} & \textbf{9.521e-01} & \textbf{9.662e-01} & \textbf{9.776e-01} \\
		\hline

	\end{tabular}}
	\end{center}
	\end{table}

\newpage

\begin{table}[h!]
\begin{center}
\caption{The maximum pointwise errors $ E^{N,M}_\varepsilon $ and the corresponding order of convergence $q^{N,M}_\varepsilon $ for the Problem 1 using the Richardson extrapolation on the hybrid scheme (\ref{3.1}) with $p=1$ and $M=N$.}
\label{tb2}
	\resizebox{\textwidth}{!}{\begin{tabular}{ c  c  c  c  c  c  c c }	
		\hline
		\textbf{$ \varepsilon \downarrow $}   & \textbf{$ N=32 $} & \textbf{$ N=64 $} & \textbf{$ N=128 $} & \textbf{$ N=256 $} & \textbf{$ N=512 $} & \textbf{$ N=1024 $} & \textbf{$ N=2048 $} \\
		\hline
		{$2^{-8}$} & 6.397e-03 & 2.258e-03 & 7.740e-04 & 2.414e-04 & 6.028e-05 & 1.507e-05 & 4.443e-06\\ 
		& 1.502e+00 & 1.545e+00 & 1.681e+00 & 2.001e+00 & 2.000e+00 & 1.762e+00\\
		\hline
		{$2^{-12}$} & 6.352e-03 & 2.246e-03 & 7.705e-04 & 2.561e-04 & 8.279e-05 & 2.612e-05 & 8.072e-06\\
		& 1.500e+00 & 1.543e+00 & 1.589e+00 & 1.629e+00 & 1.664e+00 & 1.694e+00\\ 
		\hline
		{$2^{-16}$} & 6.387e-03 & 2.248e-03 & 7.695e-04 & 2.558e-04 & 8.274e-05 & 2.611e-05 & 8.069e-06\\
		& 1.506e+00 & 1.547e+00 & 1.589e+00 & 1.629e+00 & 1.664e+00 & 1.694e+00\\
		\hline
		{$2^{-20}$} & 6.407e-03 & 2.258e-03 & 7.726e-04 & 2.563e-04 & 8.275e-05 & 2.610e-05 & 8.068e-06\\
		& 1.505e+00 & 1.547e+00 & 1.592e+00 & 1.631e+00 & 1.665e+00 & 1.694e+00\\
		\hline
		{$2^{-24}$} & 6.409e-03 & 2.261e-03 & 7.741e-04 & 2.568e-04 & 8.292e-05 & 2.613e-05 & 8.069e-06\\
		& 1.503e+00 & 1.546e+00 & 1.592e+00 & 1.631e+00 & 1.666e+00 & 1.695e+00\\
		\hline
		{$2^{-28}$} & 6.410e-03 & 2.261e-03 & 7.745e-04 & 2.570e-04 & 8.299e-05 & 2.616e-05 & 8.077e-06\\
		& 1.503e+00 & 1.546e+00 & 1.591e+00 & 1.631e+00 & 1.666e+00 & 1.695e+00\\
		\hline
		{$2^{-32}$} & 6.410e-03 & 2.262e-03 & 7.745e-04 & 2.570e-04 & 8.301e-05 & 2.616e-05 & 8.080e-06\\
		& 1.503e+00 & 1.546e+00 & 1.591e+00 & 1.631e+00 & 1.666e+00 & 1.695e+00\\
		\hline
		{$2^{-36}$} & 6.410e-03 & 2.262e-03 & 7.746e-04 & 2.571e-04 & 8.301e-05 & 2.617e-05 & 8.081e-06\\
		& 1.503e+00 & 1.546e+00 & 1.591e+00 & 1.631e+00 & 1.666e+00 & 1.695e+00\\
		\hline
		{$2^{-40}$} & 6.410e-03 & 2.262e-03 & 7.746e-04 & 2.571e-04 & 8.301e-05 & 2.617e-05 & 8.081e-06\\
		& 1.503e+00 & 1.546e+00 & 1.591e+00 & 1.631e+00 & 1.666e+00 & 1.695e+00\\
		\hline
		$ \mathbf{E^{N,M}}$ \textbf{(hybrid scheme (}\ref{3.1}) \\ \textbf{with extrapolation)} & \textbf{6.410e-03} & \textbf{2.262e-03} & \textbf{7.746e-04} & \textbf{2.571e-04} & \textbf{8.301e-05} & \textbf{2.617e-05} & \textbf{8.081e-06} \\
		\hline
		$ \mathbf{q^{N,M}}$ \textbf{(hybrid scheme} (\ref{3.1}) \\ \textbf{with extrapolation)} & \textbf{1.503e+00} & \textbf{1.546e+00} & \textbf{1.591e+00} & \textbf{1.631e+00} & \textbf{1.666e+00} & \textbf{1.695e+00} \\
		\hline
		
	\end{tabular}}
\end{center}
\end{table}

\newpage

\begin{table}[h!]
\begin{center}
\caption{The maximum pointwise errors $ E^{N,M}_\varepsilon $ and the corresponding order of convergence $q^{N,M}_\varepsilon $ for the Problem 2 using different schemes with $p=1$ and $M=N$.}
\label{tb3}
	\resizebox{\textwidth}{!}{\begin{tabular}{ c  c  c  c  c  c  c c c}
		\hline
		\textbf{$ \varepsilon \downarrow $}  & \textbf{Scheme}  & \textbf{$ N=32 $} & \textbf{$ N=64 $} & \textbf{$ N=128 $} & \textbf{$ N=256 $} & \textbf{$ N=512 $} & \textbf{$ N=1024 $} & \textbf{$ N=2048 $} \\
		\hline
		{$2^{-8}$} & \textbf{Upwind scheme on} & 6.278e-02 & 3.256e-02 & 1.716e-02 & 8.789e-03 & 4.454e-03 & 2.243e-03 & 1.126e-03\\ 
	    &\textbf{uniform mesh}& 9.474e-01 & 9.243e-01 & 9.649e-01 & 9.805e-01 & 9.896e-01 & 9.947e-01\\
		\hline
		& \textbf{Upwind scheme on} & 6.278e-02 & 3.256e-02 & 1.716e-02 & 8.789e-03 & 4.454e-03 & 2.243e-03 & 1.126e-03\\ 
		& \textbf{ Shishkin mesh} & 9.474e-01 & 9.243e-01 & 9.649e-01 & 9.805e-01 & 9.896e-01 & 9.947e-01\\
		\hline
		 & \textbf{Hybrid scheme on} & 1.468e-02 & 3.811e-03 & 1.817e-03 & 9.210e-04 & 4.637e-04 & 2.326e-04 & 1.165e-04\\ 
		& \textbf{ Shishkin mesh} & 1.946e+00 & 1.068e+00 & 9.806e-01 & 9.901e-01 & 9.950e-01 & 9.975e-01\\
		\hline
		{$2^{-12}$} & \textbf{Upwind scheme on} & 1.203e-01 & 1.131e-01 & 5.492e-02 & 2.879e-02 & 1.516e-02 & 7.794e-03 & 3.959e-03\\
		&\textbf{Uniform mesh}& 8.849e-02 & 1.043e+00 & 9.315e-01 & 9.251e-01 & 9.601e-01 & 9.773e-01\\ 
		\hline
		& \textbf{Upwind scheme on} & 1.039e-01 & 5.739e-02 & 3.595e-02 & 2.148e-02 & 1.253e-02 & 7.108e-03 & 3.959e-03\\
		& \textbf{ Shishkin mesh} & 8.570e-01 & 6.746e-01 & 7.429e-01 & 7.779e-01 & 8.177e-01 & 8.443e-01\\ 
		\hline
		& \textbf{Hybrid scheme on} & 2.615e-02 & 9.482e-03 & 3.466e-03 & 1.329e-03 & 5.531e-04 & 2.725e-04 & 1.365e-04\\
		& \textbf{ Shishkin mesh} & 1.464e+00 & 1.452e+00 & 1.382e+00 & 1.265e+00 & 1.022e+00 & 9.976e-01\\ 
		\hline
		{$2^{-16}$} & \textbf{Upwind scheme on} & 1.663e-02 & 5.109e-02 & 1.165e-01 & 1.119e-01 & 5.427e-02 & 2.849e-02 & 1.502e-02\\
		&\textbf{Uniform mesh}& -1.619e+00 & -1.189e+00 & 5.864e-02 & 1.044e+00 & 9.297e-01 & 9.232e-01\\
		\hline
		& \textbf{Upwind scheme on} & 1.023e-01 & 5.616e-02 & 3.534e-02 & 2.115e-02 & 1.236e-02 & 7.013e-03 & 3.910e-03\\
		& \textbf{ Shishkin mesh} & 8.648e-01 & 6.681e-01 & 7.405e-01 & 7.747e-01 & 8.182e-01 & 8.428e-01\\
		\hline
		& \textbf{Hybrid scheme on} & 2.646e-02 & 9.559e-03 & 3.666e-03 & 1.354e-03 & 5.732e-04 & 2.800e-04 & 1.401e-04\\
		& \textbf{ Shishkin mesh} & 1.469e+00 & 1.383e+00 & 1.437e+00 & 1.240e+00 & 1.034e+00 & 9.986e-01\\
		\hline
		{$2^{-20}$}& \textbf{Upwind scheme on} & 1.121e-02 & 6.017e-03 & 1.416e-02 & 5.130e-02 & 1.169e-01 & 1.119e-01 & 5.430e-02\\
		&\textbf{Uniform mesh}& 8.972e-01 & -1.235e+00 & -1.857e+00 & -1.188e+00 & 6.255e-02 & 1.044e+00\\
		\hline
		 & \textbf{Upwind scheme on} & 1.022e-01 & 5.612e-02 & 3.533e-02 & 2.115e-02 & 1.236e-02 & 7.010e-03 & 3.909e-03\\
		& \textbf{ Shishkin mesh} & 8.645e-01 & 6.675e-01 & 7.405e-01 & 7.747e-01 & 8.183e-01 & 8.428e-01\\
		\hline
		& \textbf{Hybrid scheme on} & 2.647e-02 & 9.530e-03 & 3.476e-03 & 1.590e-03 & 7.209e-04 & 3.157e-04 & 1.420e-04\\
		& \textbf{ Shishkin mesh} & 1.474e+00 & 1.455e+00 & 1.129e+00 & 1.141e+00 & 1.191e+00 & 1.153e+00\\
		\hline
		{$2^{-24}$} & \textbf{Upwind scheme on} & 1.121e-02 & 6.017e-03 & 3.162e-03 & 3.208e-03 & 1.496e-02 & 5.180e-02 & 1.171e-01\\
		&\textbf{Uniform mesh}& 8.972e-01 & 9.285e-01 & -2.084e-02 & -2.221e+00 & -1.792e+00 & -1.177e+00\\
		\hline
		& \textbf{Upwind scheme on} & 1.022e-01 & 5.613e-02 & 3.534e-02 & 2.115e-02 & 1.236e-02 & 7.012e-03 & 3.909e-03\\
		& \textbf{ Shishkin mesh} & 8.643e-01 & 6.675e-01 & 7.405e-01 & 7.747e-01 & 8.183e-01 & 8.428e-01\\
		\hline
		& \textbf{Hybrid scheme on} & 2.646e-02 & 9.519e-03 & 3.480e-03 & 1.335e-03 & 6.125e-04 & 3.029e-04 & 1.450e-04\\
		& \textbf{ Shishkin mesh} & 1.475e+00 & 1.452e+00 & 1.382e+00 & 1.124e+00 & 1.016e+00 & 1.063e+00\\
		\hline
		{$2^{-28}$} & \textbf{Upwind scheme on} & 1.121e-02 & 6.017e-03 & 3.162e-03 & 1.637e-03 & 8.373e-04 & 3.785e-03 & 1.528e-02\\
		&\textbf{Uniform mesh}& 8.972e-01 & 9.285e-01 & 9.499e-01 & 9.670e-01 & -2.177e+00 & -2.013e+00\\
		\hline
		& \textbf{Upwind scheme on} & 1.022e-01 & 5.614e-02 & 3.534e-02 & 2.115e-02 & 1.236e-02 & 7.012e-03 & 3.910e-03\\
		& \textbf{ Shishkin mesh} & 8.642e-01 & 6.675e-01 & 7.405e-01 & 7.748e-01 & 8.183e-01 & 8.428e-01\\
		\hline
		& \textbf{Hybrid scheme on} & 2.646e-02 & 9.515e-03 & 3.654e-03 & 1.528e-03 & 5.937e-04 & 2.816e-04 & 1.411e-04\\
		& \textbf{ Shishkin mesh} & 1.475e+00 & 1.381e+00 & 1.258e+00 & 1.364e+00 & 1.076e+00 & 9.968e-01\\
		\hline
		{$2^{-32}$} & \textbf{Upwind scheme on} & 1.121e-02 & 6.017e-03 & 3.162e-03 & 1.637e-03 & 8.373e-04 & 4.247e-04 & 8.944e-04\\
		&\textbf{Uniform mesh}& 8.972e-01 & 9.285e-01 & 9.499e-01 & 9.670e-01 & 9.791e-01 & -1.074e+00\\
		\hline
		& \textbf{Upwind scheme on} & 1.022e-01 & 5.614e-02 & 3.535e-02 & 2.115e-02 & 1.236e-02 & 7.012e-03 & 3.910e-03\\
		& \textbf{ Shishkin mesh} & 8.642e-01 & 6.675e-01 & 7.405e-01 & 7.748e-01 & 8.183e-01 & 8.428e-01\\
		\hline
		& \textbf{Hybrid scheme on} & 2.645e-02 & 9.514e-03 & 3.671e-03 & 1.472e-03 & 6.550e-04 & 2.910e-04 & 1.406e-04\\
		& \textbf{ Shishkin mesh} & 1.475e+00 & 1.374e+00 & 1.319e+00 & 1.168e+00 & 1.171e+00 & 1.049e+00\\
		\hline
		{$2^{-36}$} & \textbf{Upwind scheme on} & 1.121e-02 & 6.017e-03 & 3.162e-03 & 1.637e-03 & 8.373e-04 & 4.247e-04 & 2.143e-04\\
		&\textbf{Uniform mesh}& 8.972e-01 & 9.285e-01 & 9.499e-01 & 9.670e-01 & 9.791e-01 & 9.872e-01\\
		\hline
		& \textbf{Upwind scheme on} & 1.022e-01 & 5.614e-02 & 3.535e-02 & 2.115e-02 & 1.236e-02 & 7.012e-03 & 3.910e-03\\
		& \textbf{ Shishkin mesh} & 8.642e-01 & 6.675e-01 & 7.405e-01 & 7.748e-01 & 8.183e-01 & 8.428e-01\\
		\hline
		& \textbf{Hybrid scheme on} & 2.645e-02 & 9.514e-03 & 3.675e-03 & 1.473e-03 & 6.270e-04 & 2.924e-04 & 1.416e-04\\
		& \textbf{ Shishkin mesh} & 1.475e+00 & 1.372e+00 & 1.319e+00 & 1.232e+00 & 1.101e+00 & 1.046e+00\\
		\hline
		{$2^{-40}$} & \textbf{Upwind scheme on} & 1.121e-02 & 6.017e-03 & 3.162e-03 & 1.637e-03 & 8.373e-04 & 4.247e-04 & 2.143e-04\\
		&\textbf{Uniform mesh}& 8.972e-01 & 9.285e-01 & 9.499e-01 & 9.670e-01 & 9.791e-01 & 9.872e-01\\
		\hline
		 & \textbf{Upwind scheme on} & 1.022e-01 & 5.614e-02 & 3.535e-02 & 2.115e-02 & 1.236e-02 & 7.012e-03 & 3.910e-03\\
		& \textbf{ Shishkin mesh} & 8.642e-01 & 6.675e-01 & 7.405e-01 & 7.748e-01 & 8.183e-01 & 8.428e-01\\
		\hline
		& \textbf{Hybrid scheme on} & 2.645e-02 & 9.514e-03 & 3.677e-03 & 1.473e-03 & 6.270e-04 & 2.865e-04 & 1.417e-04\\
		& \textbf{ Shishkin mesh} & 1.475e+00 & 1.372e+00 & 1.320e+00 & 1.232e+00 & 1.130e+00 & 1.016e+00\\
		\hline
		{$ \mathbf{E^{N,M}}$} &\textbf{Upwind scheme on} & \textbf{1.203e-01} & \textbf{1.131e-01} & \textbf{1.165e-01} & \textbf{1.119e-01} & \textbf{1.169e-01} & \textbf{1.119e-01} & \textbf{1.171e-01} \\
		&\textbf{Uniform mesh} & \textbf{8.849e-02} & \textbf{-4.264e-02} & \textbf{5.864e-02} & \textbf{-6.339e-02} & \textbf{6.255e-02}  & \textbf{-6.519e-02} \\
		\hline
		{$ \mathbf{E^{N,M}}$} &\textbf{Upwind scheme } & \textbf{1.039e-01} & \textbf{5.739e-02} & \textbf{3.595e-02} & \textbf{2.148e-02} & \textbf{1.253e-02} & \textbf{7.108e-03} & \textbf{3.959e-03} \\
		& \textbf{ Shishkin mesh} & \textbf{8.570e-01} & \textbf{6.746e-01} & \textbf{7.429e-01} & \textbf{7.779e-01} & \textbf{8.177e-01}  & \textbf{8.443e-01} \\
		\hline
		{$ \mathbf{E^{N,M}}$}& \textbf{ hybrid scheme on} & \textbf{2.647e-02} & \textbf{9.559e-03} & \textbf{3.677e-03} & \textbf{1.590e-03} & \textbf{7.209e-04} & \textbf{3.157e-04} & \textbf{1.450e-04} \\
		&\textbf{Shishkin mesh} (\ref{3.1}) & \textbf{1.470e+00} & \textbf{1.378e+00} & \textbf{1.209e+00} & \textbf{1.141e+00} & \textbf{1.191e+00} & \textbf{1.122e+00} \\
		\hline
		
	\end{tabular}}
	\end{center}
	\end{table}

\newpage
 
\begin{table}[h!]
\begin{center}
\caption{The maximum pointwise errors $ E^{N,M}_\varepsilon $ and the corresponding order of convergence $q^{N,M}_\varepsilon $ for the Problem 2 using the Richardson extrapolation on the hybrid scheme (\ref{3.1}) with $p=1$ and $M=N$.}
\label{tb4}
	\resizebox{\textwidth}{!}{\begin{tabular}{ c  c  c  c  c  c  c c }
		\hline
		\textbf{$ \varepsilon \downarrow $}   & \textbf{$ N=32 $} & \textbf{$ N=64 $} & \textbf{$ N=128 $} & \textbf{$ N=256 $} & \textbf{$ N=512 $} & \textbf{$ N=1024 $} & \textbf{$ N=2048 $} \\
		\hline
		{$2^{-8}$} & 1.347e-02 & 3.252e-03 & 8.099e-04 & 2.023e-04 & 5.054e-05 & 1.999e-05 & 9.089e-06\\ 
		& 2.051e+00 & 2.005e+00 & 2.001e+00 & 2.001e+00 & 1.338e+00 & 1.137e+00\\
		\hline
		{$2^{-12}$} & 2.301e-02 & 8.141e-03 & 2.716e-03 & 9.060e-04 & 2.917e-04 & 9.194e-05 & 2.838e-05\\
		& 1.499e+00 & 1.584e+00 & 1.584e+00 & 1.635e+00 & 1.666e+00 & 1.696e+00\\ 
		\hline
		{$2^{-16}$} & 2.298e-02 & 8.123e-03 & 2.711e-03 & 9.035e-04 & 2.910e-04 & 9.170e-05 & 2.831e-05\\
		& 1.500e+00 & 1.583e+00 & 1.585e+00 & 1.634e+00 & 1.666e+00 & 1.696e+00\\
		\hline
		{$2^{-20}$} & 2.299e-02 & 8.126e-03 & 2.712e-03 & 9.031e-04 & 2.909e-04 & 9.166e-05 & 2.830e-05\\
		& 1.500e+00 & 1.583e+00 & 1.586e+00 & 1.634e+00 & 1.666e+00 & 1.696e+00\\
		\hline
		{$2^{-24}$} & 2.299e-02 & 8.127e-03 & 2.713e-03 & 9.031e-04 & 2.909e-04 & 9.165e-05 & 2.829e-05\\
		& 1.500e+00 & 1.583e+00 & 1.587e+00 & 1.634e+00 & 1.666e+00 & 1.696e+00\\
		\hline
		{$2^{-28}$} & 2.299e-02 & 8.127e-03 & 2.713e-03 & 9.031e-04 & 2.909e-04 & 9.165e-05 & 2.829e-05\\
		& 1.500e+00 & 1.583e+00 & 1.587e+00 & 1.634e+00 & 1.666e+00 & 1.696e+00\\
		\hline
		{$2^{-32}$} & 2.299e-02 & 8.127e-03 & 2.713e-03 & 9.031e-04 & 2.909e-04 & 9.165e-05 & 2.829e-05\\
		& 1.500e+00 & 1.583e+00 & 1.587e+00 & 1.634e+00 & 1.666e+00 & 1.696e+00\\
		\hline
		{$2^{-36}$} & 2.299e-02 & 8.127e-03 & 2.713e-03 & 9.031e-04 & 2.909e-04 & 9.165e-05 & 2.829e-05\\
		& 1.500e+00 & 1.583e+00 & 1.587e+00 & 1.634e+00 & 1.666e+00 & 1.696e+00\\
		\hline
		{$2^{-40}$} & 2.299e-02 & 8.127e-03 & 2.713e-03 & 9.031e-04 & 2.909e-04 & 9.165e-05 & 2.829e-05\\
		& 1.500e+00 & 1.583e+00 & 1.587e+00 & 1.634e+00 & 1.666e+00 & 1.696e+00\\
		\hline
		$ \mathbf{E^{N,M}}$ \textbf{(hybrid scheme (}\ref{3.1}) \\ \textbf{with extrapolation)} & \textbf{2.301e-02} & \textbf{8.141e-03} & \textbf{2.716e-03} & \textbf{9.060e-04} & \textbf{2.917e-04} & \textbf{9.194e-05} & \textbf{2.838e-05} \\
		\hline
		$ \mathbf{q^{N,M}}$ \textbf{(hybrid scheme} (\ref{3.1}) \\ \textbf{with extrapolation)} & \textbf{1.499e+00} & \textbf{1.584e+00} & \textbf{1.584e+00} & \textbf{1.635e+00} & \textbf{1.666e+00} & \textbf{1.696e+00} \\
		\hline
		
	\end{tabular}
}\end{center}
\end{table}

\begin{table}[h!]
\begin{center}
\caption{The $\varepsilon$-uniform maximum pointwise errors $ E^{N,M} $ and the corresponding order of convergence $q^{N,M} $ for the Problem 1 using the Richardson extrapolation on the hybrid scheme (\ref{3.1}) for different values of $p$ and $M=N$.}
\label{tb5}
	\resizebox{\textwidth}{!}{\begin{tabular}{ c  c  c  c  c  c  c c }
		\hline
		\textbf{$ p \downarrow $}   & \textbf{$ N=32 $} & \textbf{$ N=64 $} & \textbf{$ N=128 $} & \textbf{$ N=256 $} & \textbf{$ N=512 $} & \textbf{$ N=1024 $} & \textbf{$ N=2048 $} \\
		\hline 
		{$3$} & 1.014e-02 & 3.724e-03 & 1.293e-03 & 4.318e-04 & 1.396e-04 & 4.402e-05 & 1.360e-05\\ 
		& 1.445e+00 & 1.526e+00 & 1.582e+00 & 1.629e+00 & 1.665e+00 & 1.695e+00\\
		\hline
		{$6$} & 1.183e-02 & 4.284e-03 & 1.515e-03 & 6.023e-04 & 2.183e-04 & 7.572e-05 & 2.527e-05\\
		& 1.465e+00 & 1.500e+00 & 1.331e+00 & 1.464e+00 & 1.528e+00 & 1.583e+00\\ 
		\hline
		
	\end{tabular}
}\end{center}
\end{table}

\begin{table}[h!]
\begin{center}
\caption{The $\varepsilon$-uniform maximum pointwise errors $ E^{N,M} $ and the corresponding order of convergence $q^{N,M} $ for the Problem 2 using the Richardson extrapolation on the hybrid scheme (\ref{3.1}) for different values of $p$ and $M=N$.}
\label{tb6}
	\resizebox{\textwidth}{!}{\begin{tabular}{ c  c  c  c  c  c  c c }
		\hline
		\textbf{$ p \downarrow $}   & \textbf{$ N=32 $} & \textbf{$ N=64 $} & \textbf{$ N=128 $} & \textbf{$ N=256 $} & \textbf{$ N=512 $} & \textbf{$ N=1024 $} & \textbf{$ N=2048 $} \\
		\hline 
		{$2$} & 1.347e-02 & 4.914e-03 & 1.714e-03 & 5.728e-04 & 1.852e-04 & 5.843e-05 & 1.805e-05\\ 
		& 1.454e+00 & 1.519e+00 & 1.582e+00 & 1.629e+00 & 1.664e+00 & 1.695e+00\\
		\hline
		{$5$} & 2.752e-02 & 1.070e-02 & 3.766e-03 & 1.259e-03 & 4.080e-04 & 1.287e-04 & 3.977e-05\\
		& 1.363e+00 & 1.506e+00 & 1.581e+00 & 1.626e+00 & 1.664e+00 & 1.695e+00\\ 
		\hline
		
	\end{tabular}
}\end{center}
\end{table}

\end{document}